\documentclass[11pt,a4paper,reqno]{amsart}
\usepackage{amsfonts,caption,amsmath,amssymb,enumerate,latexsym,xcolor,amsthm,multicol,cases,tikz,empheq}
\usepackage{hyperref}
\usepackage{xcolor}
\usepackage[normalem]{ulem}
\usepackage{array} 
\usetikzlibrary{calc}

\newcommand\AGL{\mathrm{AGL}}\newcommand\Aut{\mathrm{Aut}}

\newcommand\Cay{\mathrm{Cay}}

\newcommand\FF{\mathbb{F}}\newcommand\Fun{\mathrm{Fun}}
\newcommand\GaL{\mathrm{\Gamma L}}
\newcommand\GL{\mathrm{GL}}

\newcommand\la{\langle}

\newcommand\NN{\mathsf{N}}
\newcommand\Out{\mathrm{Out}}
\newcommand\PGaL{\mathrm{P\Gamma L}}
\newcommand\PGL{\mathrm{PGL}}
\newcommand\PSiL{\mathrm{P\Sigma L}}\newcommand\PSL{\mathrm{PSL}}

\newcommand\ra{\rangle}
\newcommand\SL{\mathrm{SL}}
\newcommand\Soc{\mathrm{Soc}}
\newcommand\Sym{\mathrm{Sym}}
\newcommand\Syl{\mathrm{Syl}}

\newcommand\Z{\mathbb{Z}}\newcommand\ZZ{\mathsf{Z}}

\newtheorem{theorem}{Theorem}[section]
\newtheorem{corollary}[theorem]{Corollary}
\newtheorem{lemma}[theorem]{Lemma}

\newtheorem{hypothesis}[theorem]{Hypothesis}
\theoremstyle{definition}

\newtheorem{example}[theorem]{Example}
\newtheorem{question}[theorem]{Question}

\newtheorem{notation}[theorem]{Notation}

\newtheorem*{remark}{Remark}

\begin{document}
\title[Which maximal subgroups are perfect codes?]{Which maximal subgroups are perfect codes?}

\author[S.Qiao]{Shouhong Qiao}
\address[Shouhong Qiao]{School of Mathematics and Statistics, Guangdong University  of Technology,  Guangzhou 510520, P. R. China}
\email{qshqsh513@163.com}

\author[N.Su]{Ning Su}
\address[Ning Su]{School of Mathematics, Sun Yat-sen University, Guangzhou 510275, P. R. China}
\email{suning3@mail.sysu.edu.cn}

\author[B.Xia]{Binzhou Xia}
\address[Binzhou Xia]{School of Mathematics and Statistics, The University of Melbourne, Parkville, VIC 3010, Australia}
\email{binzhoux@unimelb.edu.au}

\author[Z.Zhang]{Zhishuo Zhang}
\address[Zhishuo Zhang]{School of Mathematics and Statistics, The University of Melbourne, Parkville, VIC 3010, Australia}
\email{zhishuoz@student.unimelb.edu.au}

\author[S.Zhou]{Sanming Zhou}
\address[Sanming Zhou]{School of Mathematics and Statistics, The University of Melbourne, Parkville, VIC 3010, Australia}
\email{sanming@unimelb.edu.au}

\begin{abstract}
A perfect code in a graph $\Gamma=(V, E)$ is a subset $C$ of $V$ such that no two vertices in $C$ are adjacent and every vertex in $V{\setminus}C$ is adjacent to exactly one vertex in $C$. A subgroup $H$ of a group $G$ is called a subgroup perfect code of $G$ if it is a perfect code in some Cayley graph of $G$. In this paper, we undertake a systematic study of which maximal subgroups of a group can be perfect codes. Our approach highlights a characterization of subgroup perfect codes in terms of their ``local'' complements.
\end{abstract}

\maketitle

\section{Introduction}

Given a simple and undirected graph $\Gamma$ with vertex set $V$, a subset $C\subseteq V$ is called a \emph{perfect code} in $\Gamma$ if $C$ is an independent set and every vertex in $V\setminus C$ is adjacent to exactly one vertex in $C$. A perfect code is also known as an \emph{efficient dominating set}~\cite{DS2003} or \emph{independent perfect dominating set}~\cite{L2001}.

The investigation of perfect codes in graphs began with Biggs' extension of classical perfect codes, such as Hamming codes and Lee codes, to distance-transitive graphs~\cite{B1973}. As a matter of fact, a perfect code in the Hamming graph $H(n,q)$ is exactly a $q$-ary $1$-error correcting perfect code of length $n$ under the Hamming distance. As justified in~\cite{HBZ2018}, since $H(n,q)$ is a Cayley graph, perfect codes in Cayley graphs can be considered as a generalization of $1$-error correcting perfect codes in coding theory, and they are also related to tilings of the underlying groups. Because of these, perfect codes in Cayley graphs have attracted significant attention in recent years, leading to numerous results for Cayley graphs over various groups, ranging from cyclic and dihedral groups to more complex families of groups~\cite{D2014,DSLW2017,FHZ2017,HBZ2018,MBG2007,OPR2007,RM2013,Z2019}.

A natural and compelling direction is to consider perfect codes that are themselves subgroups of a given group $G$~\cite{BLZ2025,CLZ2025,CM2013,CWX2020,KAK2023,MWWZ2020,TM2013,WZ2023,XZZ2024,Z2023,ZZ2022}. Notably, Hamming codes, being linear subspaces, constitute one of the most important families of subgroup perfect codes. The additional subgroup structure enriches the theoretical framework and offers practical advantages, particularly in terms of efficient representation and computation.

In this paper, all groups are finite. For a group $G$ with identity element $e$ and an inverse-closed subset $S$ of $G\setminus \{e\}$, the \emph{Cayley graph} $\Cay(G,S)$ is the graph with vertex set $G$ and edge set $\{\{g,sg\}\mid s\in S,\; g\in G\}$. To classify the pair $(H,S)$ such that $H$ is a subgroup perfect code of $\Cay(G,S)$, a natural starting point is to determine which subgroups $H\leq G$ admit such a pair. To this end, Huang, Xia, and Zhou~\cite{HBZ2018} introduced the following concept in 2018: A subset $C$ of a group $G$ is called a \emph{perfect code of $G$} if there exists a Cayley graph of $G$ which admits $C$ as a perfect code. In particular, if $C$ is also a subgroup of $G$, then $C$ is referred to as a \emph{subgroup perfect code of $G$}. They established a purely group-theoretic necessary and sufficient condition for a normal subgroup $A$ of $G$ to be a perfect code, which has since inspired active research on subgroup perfect codes of groups~\cite{CWX2020,KAK2023,MWWZ2020,WZ2023,ZZ2022}. In 2020, Chen, Wang and Xia further extended the characterization for normal subgroups to a more general criterion that determines when an arbitrary subgroup of $G$ is a perfect code as follows.

\begin{theorem}[{\cite[Theorem~1.2]{CWX2020}}]\label{thm:1}
    Let $G$ be a group and $H\leq G$. Then the following statements are equivalent:
    \begin{enumerate}[\rm(a)]
        \item \label{enu:pc1} $H$ is a perfect code of $G$;
        \item \label{enu:pc2} there exists an inverse-closed left transversal of $H$ in $G$;
        \item \label{enu:pc3} for each $a\in G$ such that $a^2\in H$ and $|H|/|H\cap H^a|$ is odd, there exists $b\in aH$ such that $b^2=e$;
        \item \label{enu:pc4} for each $a\in G$ such that $HaH=Ha^{-1}H$ and $|H|/|H\cap H^a|$ is odd, there exists $b\in aH$ such that $b^2=e$.
    \end{enumerate}
\end{theorem}

\begin{remark}
    Since the left coset convention is commonly used in the literature, we state the left coset version of~\cite[Theorem~1.2]{CWX2020} in Theorem~\ref{thm:1}. We also note that the equivalence between~\eqref{enu:pc1} and~\eqref{enu:pc4} was previously established in~\cite{ZZ2022}.
\end{remark}

Observe that for a subgroup $H$ of a group $G$, if $H$ is a perfect code of $G$, then it is also a perfect code of any subgroup of $G$ that contains $H$ (see Lemma~\ref{lm:quotient}\eqref{enu:2.2a}). Thus, in determining whether $H$ is a perfect code of $G$, it is natural to begin by examining whether $H$ is a perfect code of the smallest overgroup properly containing it. In other word, we focus on the case when $H$ is a maximal subgroup of $G$. From a coding-theoretic perspective, a good perfect code should have large cardinality in order to make efficient use of the channel. Given that maximal subgroups of the symmetric group $S_n$ form a rich and extensively studied class, the authors of~\cite{XZZ2024} initiated the study of perfect codes arising from such subgroups. Building on these motivations, we consider the following, more general question in the present paper.

\begin{question}\label{ques:1}
Given a group $G$ and a maximal subgroup $M$ of $G$, is $M$ a perfect code of $G$?
\end{question}

In this paper, we develop a systematic framework to address this question. In parallel, we introduce a new approach to studying subgroup perfect codes, prompted by a new perspective we discover through an equivalent formulation (see Theorem~\ref{thm:pciff}) of Theorem~\ref{thm:1}. In fact, Theorem~\ref{thm:pciff} reveals that whether a subgroup $H$ is a perfect code of a group $G$ depends on whether $H$ admits ``local complements". Unlike the existing literature, which focuses primarily on the structure of the ambient group $G$, our new perspective highlights the role of the internal structure of the subgroup $H$ in determining perfect code properties. We will see that this shift in perspective appears to be more effective in identifying subgroup perfect codes.

Building on Theorem~\ref{thm:pciff}, we demonstrate how this new viewpoint can be applied to establish necessary and sufficient conditions for a group $H$ with certain structure to be a perfect code in an arbitrary overgroup $G$. These results will prove particularly useful in the analysis of concrete examples in Sections~\ref{sec:4} and~\ref{sec:6}--\ref{sec:7}. We also remark that similar techniques can be employed to obtain analogous characterizations for a broad range of groups.

From Section~\ref{sec:4}, we start to develop our general framework for addressing Question~\ref{ques:1}, namely, whether a maximal subgroup $M$ of a given group $G$ is a perfect code of $G$. The idea is as follows. Let $K=\mathrm{Core}_G(M)$. A necessary condition for $M$ to be a perfect code of $G$ is that the quotient $M/K$ is a perfect code of the quotient group $G/K$ (see Lemma~\ref{lm:quotient}\eqref{enu:quotient}). Therefore, in order to study Question~\ref{ques:1}, it is natural to begin by examining whether $M/K$ is a perfect code of $G/K$, which will be the focus of Sections~\ref{sec:5}--\ref{sec:7}. In Section~\ref{sec:4}, we study the lifting problem: assuming that $M/K$ is a perfect code of $G/K$, does it follow that $M$ is also a perfect code of $G$? We show that the answer is affirmative when $G$ is a split extension of $K$ by $G/K$, by establishing a useful result we call the \emph{Diamond Lemma}. However, the situation becomes more intricate in the non-split case, as illustrated by two examples.

From Section~\ref{sec:5}, we examine the problem of whether $M/K$ is a perfect code of $G/K$. A permutation group is said to be \emph{primitive} if it preserves no non-trivial partition. The right multiplication action of $G$ on $G/M$, the set of right cosets of $M$ in $G$, has kernel $K$, and thus $G/K$ acts primitively on $G/M$ with point stabilizer $M/K$. Accordingly, we are essentially led to the following question.

\begin{question}\label{ques:quasi}
    Given a primitive group $G$ with point stabilizer $M$, is $M$ a perfect code of $G$?
\end{question}

Primitive groups have been classified into eight types: $\mathrm{HA}$, $\mathrm{HS}$, $\mathrm{HC}$, $\mathrm{TW}$, $\mathrm{SD}$, $\mathrm{CD}$, $\mathrm{AS}$ and $\mathrm{PA}$ (see~\cite{P1997} for example). The main result of Section~\ref{sec:5} is the following theorem, which answers Question~\ref{ques:quasi} for all primitive groups other than the ones of type $\mathrm{AS}$ and $\mathrm{PA}$.

\begin{theorem}\label{thm:6}
    Let $G$ be a primitive group of type $\mathrm{HA}$, $\mathrm{HS}$, $\mathrm{HC}$, $\mathrm{TW}$, $\mathrm{SD}$ or $\mathrm{CD}$, and let $M$ be a point stabilizer. Then $M$ is a perfect code of $G$.
\end{theorem}

A permutation group is called \emph{quasiprimitive} if every non-trivial normal subgroup is transitive. As with primitive groups, every quasiprimitive group belongs to exactly one of the eight types mentioned above (see~\cite{P1997} for further details). Notably, Theorem~\ref{thm:6} also holds for quasiprimitive groups of the six types listed above (see Lemmas~\ref{lm:4types} and~\ref{lm:2types}). These results are also derived using the Diamond Lemma established in Section~\ref{sec:4}.

In Section~\ref{sec:6}, we investigate subgroup perfect codes in primitive groups of type~$\mathrm{AS}$. Several related results can be found in the literature~\cite{BLZ2025,CLZ2025,MWWZ2020,XZZ2024,Z2023}. In particular, the question of whether a maximal subgroup of $\PSL_2(q)$ forms a perfect code has been investigated in~\cite{MWWZ2020,Z2023}, and a complete classification is finally obtained in~\cite{CLZ2025}. Given the diversity and complexity of almost simple groups, a full classification of subgroup perfect codes in this setting appears infeasible. Therefore, we focus on extending the classification of maximal subgroup perfect codes in $\PSL_2(q)$ to the broader class of almost simple groups with socle $\PSL_2(q)$, as stated below, where the proof showcases the strength of the $2$-local analysis developed in Section~\ref{sec:3} as well as the Diamond Lemma in Section~\ref{sec:4}.

\begin{theorem}\label{thm:AS}
    Let $T=\PSL_2(q)$ with prime power $q\geq 4$, let $G$ be a primitive almost simple group with socle $T$, and let $M$ be the point stabilizer of $G$. Then $M$ is not a perfect code of $G$ if and only if one of the following holds:
    \begin{enumerate}[\rm(a)]
        \item \label{enu:1.5a} $q>7$, $q\equiv -1\pmod{8}$, $|G/T|$ is odd, and $M\cap T\cong D_{q-1}$;
        \item \label{enu:1.5b} $q>9$, $q\equiv 1\pmod{8}$, $|G/T|$ is odd, and $M\cap T\cong D_{q+1}$;
        \item \label{enu:1.5c} $q\equiv 1\pmod{8}$, $|G/T|\equiv 2\pmod{4}$, $G\not \leq \PSiL_2(q)$, $G\not \geq \PGL_2(q)$, and $M\cap T\cong D_{q+1}$.
    \end{enumerate}
\end{theorem}

We observe from this result that certain almost simple groups, such as $\PGL_2(q)$, have the property that every maximal subgroup is a perfect code. This motivates us to pose the following question.

\begin{question}\label{ques:AS}
    Which almost simple groups have the property that every maximal subgroup is a perfect code?
\end{question}

As a special case of Question~\ref{ques:AS}, it is asked in~\cite[Question~1.9]{XZZ2024} that whether every maximal subgroup of $S_n$ is a perfect code. So far, the only results on this question are for intransitive maximal subgroups of $S_n$ and the affine groups $\AGL_1(p)$ as maximal subgroups of $S_p$ for odd primes $p$ (see~\cite[Propositions~4.1 and~4.3]{XZZ2024}). In Section~\ref{sec:6}, we also give the following result.

\begin{theorem}\label{thm:perfectcode}
    Let $q$ be a prime power such that $q\equiv 3\pmod{4}$, let $G=S_{q^2}$, and let $H=\AGL_2(q)< G$. Then $H$ is a perfect code of $G$.
\end{theorem}

In this case when $q$ is a prime, $\AGL_2(q)$ is a maximal subgroup of $S_{q^2}$ (see~\cite{LPS1987} for instance), and thus Theorem~\ref{thm:perfectcode} provides a further contribution toward addressing~\cite[Question~1.9]{XZZ2024}. The proof of Theorem~\ref{thm:perfectcode} combines our new approach from Section~\ref{sec:3}, which emphasizes the role of local subgroup structure, with existing techniques in the literature that focus on the global structure of the group. This interplay illustrates the inherent complexity of addressing Question~\ref{ques:quasi} for all primitive groups of type~$\mathrm{AS}$.

Finally, in Section~\ref{sec:7}, we focus on primitive groups of type $\mathrm{PA}$. These groups can be viewed as ``blow-ups" of primitive groups of type $\mathrm{AS}$---that is, given a primitive group $G$ of type $\mathrm{PA}$, there exists some primitive group $H$ of type $\mathrm{AS}$ and positive integer $n$ such that $G=H\wr S_n$, where the point stabilizer of $G$ is $M\wr S_n$ for some point stabilizer $M$ of $H$. We prove that if $M$ is a perfect code of $H$, then $M\wr S_n$ is a perfect code of $G$. In other words, if the point stabilizer of a primitive group of type~$\mathrm{AS}$ is a perfect code, then the point stabilizer of the corresponding primitive group of type~$\mathrm{PA}$ is also a perfect code. However, we show that the converse does not hold by presenting a general result and a family of examples, demonstrating that the study of Question~\ref{ques:quasi} for groups of type~$\mathrm{PA}$ cannot be reduced to the case of type~$\mathrm{AS}$ in a straightforward manner.

\section{Preliminaries}\label{sec:2}

Let $G$ be a group and $n$ be a positive integer. Denote the identity element of $G$ by $e$, the center of $G$ by $\ZZ(G)$, and denote by $\NN_G(H)$ the normalizer of $H$ in $G$ for $H\leq G$. The \emph{right regular permutation representation} $R(G)$ of $G$ is the subgroup of $\Sym(G)$ consisting of $R(g)$ with $g\in G$, where
\[
    R(g)\colon G\to G,\ \ h\mapsto hg.
\]
For subgroups $H$ and $K$ of $G$, denote $HK=\{hk\mid h\in H,\; k\in K\}$. For $x, y\in G$, denote $x^y=y^{-1}xy$ and $[x,y]=x^{-1}y^{-1}xy$. We use $1$ to denote the trivial group, use $D_{2n}$ to denote the dihedral group of order $2n$ if $n\geq 2$ (with $D_4$ being $C_2\times C_2$), and use $[n]$ to represent the set $\{1,\ldots, n\}$. For a prime $p$, denote by $\Syl_p(G)$ the set of all Sylow $p$-subgroups of $G$, denote by $n_p$ the largest $p$-power dividing $n$, and write $n_{p'}=n/n_p$.

The following three lemmas collect some properties of perfect codes from~\cite[Theorems~1.1 and~1.2]{Z2023},~\cite[Lemma~2.2 and Theorem~3.7]{ZZ2022} and~\cite[Proposition~3.4]{Z2023}, respectively.

\begin{lemma}\label{lm:Z2023}
    Let $G$ be a group and $H$ a subgroup of $G$, and let $Q\in \Syl_2(H)$. Then the following statements are equivalent:
    \begin{enumerate}[\rm(a)]
      \item $H$ is a perfect code of $G$;
      \item $Q$ is a perfect code of $G$;
      \item $Q$ is a perfect code of any Sylow $2$-subgroup of $\NN_G(Q)$;
      \item every Sylow $2$-subgroup of $H$ is a perfect code of $G$.
    \end{enumerate}
\end{lemma}

\begin{lemma}\label{lm:quotient}
Let $G$ be a group, let $H$ be a subgroup of $G$, and let $N$ be a normal subgroup of $G$ contained in $H$. Then the following statements hold:
\begin{enumerate}[\rm(a)]
    \item \label{enu:2.2a} if $H$ is a perfect code of $G$ and $H\leq M\leq G$, then $H$ is a perfect code of $M$;
    \item \label{enu:quotient} if $H$ is a perfect code of $G$, then $H/N$ is a perfect code of $G/N$;
    \item \label{enu:quotient3} if $N$ and $H/N$ are perfect codes of $G$ and $G/N$, respectively, then $H$ is a perfect code of $G$.
\end{enumerate}
\end{lemma}

\begin{lemma}\label{lm:elem}
    Let $G$ be a group with elementary abelian Sylow $2$-subgroups. Then every subgroup $H$ of $G$ is a perfect code of $G$.
\end{lemma}

The following result is well known and will be used repeatedly in this paper.

\begin{lemma}[{\cite[8.2~Satz]{H1967}}]\label{lm:2group}
    A $2$-group has a unique involution if and only if it is either cyclic or a generalized quaternion group.
\end{lemma}

The remainder of this section is devoted to presenting some basic facts about projective linear groups that will be used in Sections~\ref{sec:6} and~\ref{sec:7}.

The following lemma provides the structure of Sylow $2$-subgroups of $\PGL_2(q)$ and $\PSL_2(q)$. It is well known and can be read off from~\cite[\S 7 and \S 8]{H1967} for instance.

\begin{lemma}\label{lm:Sylow2PSL}
    Let $q$ be a prime power, let $G\in \Syl_2(\PGL_2(q))$, and let $H\in\Syl_2(\PSL_2(q))$. Then the following statements hold:
    \begin{enumerate}[\rm(a)]
        \item if $q=2^f$ for some integer $f$, then $G\cong C_2^f\cong H$;
        \item if $q\equiv 1\pmod{4}$, then $G\cong D_{2(q-1)_2}$ and $H\cong D_{(q-1)_2}$;
        \item if $q\equiv 3\pmod{4}$, then $G\cong D_{2(q+1)_2}$ and $H\cong D_{(q+1)_2}$.
    \end{enumerate}
\end{lemma}

The following lemma gives tables of the maximal subgroups of $\PSL_2( q)$ and $\PGL_2(q)$, which can be read off from~\cite[Tables~8.1 and~8.2]{BHD2013}. Based on the structure of each maximal subgroup, its Sylow $2$-subgroup can be determined, as given in the tables (with the help of Lemma~\ref{lm:Sylow2PSL} for some rows).

\begin{lemma}\label{lm:MaximalPSL}
Let $q = p^f \geq 5$ for some odd prime $p$ and positive integer $f$. Then the maximal subgroups $M$ of $\PSL_2(q)$ and $\PGL_2(q)$ are described in Tables~$\ref{tab:M}$ and~$\ref{tab:MPGL}$, respectively.
\end{lemma}

\begin{table}[h]
    \centering
    \renewcommand{\arraystretch}{1.5}
    \begin{tabular}{c|c|c}
    \noalign{\hrule height 1.2pt}
    $M$ & Sylow $2$-subgroup of $M$ & Conditions \\
    \noalign{\hrule height 1.2pt}
    $C_p^f \rtimes C_{(q-1)/2}$ & $C_{(q-1)_2/2}$ & — \\
    \hline
    $D_{q-1}$ &
    \rule{0pt}{4ex}$\begin{aligned}
        &D_{(q-1)_2} &&\text{if }q\equiv1\pmod4\\
        &C_2 &&\text{if }q\equiv 3\pmod4
    \end{aligned}$ & $q \ge 13$ \\
    \hline
    $D_{q+1}$ &
    \rule{0pt}{4ex}$\begin{aligned}
        &D_{(q+1)_2} &&\text{if }q\equiv 3\pmod4\\
        &C_2 &&\text{if }q\equiv 1\pmod4
    \end{aligned}$ & $q\neq7,9$ \\
    \hline
    $\PGL_2(q_0)$ &
    \rule{0pt}{4ex}$\begin{aligned}
        &D_{2(q_0-1)_2} &&\text{if }q_0\equiv1\pmod4\\
        &D_{2(q_0+1)_2} &&\text{if }q_0\equiv 3\pmod4
    \end{aligned}$ & $q = q_0^2$ \\
    \hline
    $\PSL_2(q_0)$ &
    \rule{0pt}{4ex}$\begin{aligned}
        &D_{(q_0-1)_2} &&\text{if }q_0\equiv1\pmod4\\
        &D_{(q_0+1)_2} &&\text{if }q_0\equiv 3\pmod4
    \end{aligned}$ & $q = q_0^r$ with $r$ odd prime \\
    \hline
    $A_5$ & $D_8$ &
    \rule{0pt}{4ex}$\begin{aligned}
        &q=p\equiv\pm1\pmod{10} \text{ or}\\
        &q=p^2 \text{ with }p\equiv\pm3\pmod{10}
    \end{aligned}$ \\
    \hline
    $A_4$ & $D_4$ & $q=p\equiv\pm3, 5, \pm 13\pmod{40}$ \\
    \hline
    $S_4$ & $D_8$ & $q=p\equiv\pm1\pmod8$ \\
    \noalign{\hrule height 1.2pt}
    \end{tabular}
    \caption{Maximal subgroups of $\PSL_2(q)$ and their Sylow $2$-subgroups}
    \label{tab:M}
\end{table}

\begin{table}[h]
    \centering
    \renewcommand{\arraystretch}{1.5}
    \begin{tabular}{c|c|c}
    \noalign{\hrule height 1.2pt}
    $M$ & Sylow $2$-subgroup of $M$ & Conditions \\
    \noalign{\hrule height 1.2pt}
    $C_p^f \rtimes C_{q-1}$ & $C_{(q-1)_2}$ & $q=p^f$ \\
    \hline
    $D_{2(q-1)}$ & $D_{2(q-1)_2}$ & $q >5$ \\
    \hline
    $D_{2(q+1)}$ & $D_{2(q+1)_2}$ & \\
    \hline
    $\PGL_2(q_0)$ &
    \rule{0pt}{4ex}$\begin{aligned}
        &D_{2(q_0-1)_2} &&\text{if }q_0\equiv1\pmod4,\\
        &D_{2(q_0+1)_2} &&\text{if }q_0\equiv 3\pmod4
    \end{aligned} $ & $q = q_0^r$ \\
    \hline
    $\PSL_2(q)$ &
    \rule{0pt}{4ex}$\begin{aligned}
        &D_{(q-1)_2} &&\text{if }q\equiv1\pmod4,\\
        &D_{(q+1)_2} &&\text{if }q\equiv 3\pmod4
    \end{aligned}$ & \\
    \hline
    $S_4$ & $D_8$ & $q\equiv\pm 3\pmod8$ \\
    \noalign{\hrule height 1.2pt}
    \end{tabular}
    \caption{Maximal subgroups of $\PGL_2(q)$ and their Sylow $2$-subgroups}
    \label{tab:MPGL}
\end{table}

The following result is well-known. Although it can be proved by standard linear-algebraic arguments, we give a more group-theoretic proof for the sake of self-containment and out of our own interest.

\begin{lemma}\label{lm:involutionPSL}
    For each prime power $q$, all involutions in $\PSL_2(q)$ are conjugate.
\end{lemma}

\begin{proof}
    Let $G=\PSL_2(q)$. If $q$ is even, then a Sylow $2$-subgroup of $G=\SL_2(q)$ is contained in the subgroup $\AGL_1(q)$ of $G$, where all involutions are conjugate. Now assume that $q$ is odd. Write $|G|=2^kr$ for some integer $k\geq 2$ and odd integer $r\geq 1$, and let $P\in \Syl_2(G)$. By Lemma~\ref{lm:Sylow2PSL}, $P\cong D_{2^k}$. Take a cyclic subgroup $C$ of $P$ of index $2$. It suffices to show that each involution in $G$ is conjugate to an element of $C$ in $G$.

    Let $\Omega$ be the set of right cosets of $C$ in $G$, and let  $\varphi: G\to \Sym(\Omega)$ be the right multiplication action of $G$ on $\Omega$. Since $|\Omega|=2|G|/|P|$, we have $|\Omega|=2(2\ell+1)$ for some integer $\ell\geq 0$. Take an arbitrary involution $x\in G\setminus C$. Suppose for a contradiction that $x^\varphi$ has no fixed point in $\Omega$. Since $x^\varphi\neq e$, it follows that $x^\varphi$ is the product of $2\ell+1$ disjoint transpositions and hence an odd permutation. However, $G$ has no subgroup of index $2$, a contradiction. Therefore, $x^\varphi$ has a fixed point, say, $Cg$ for some $g\in G$. This means that $Cgx=(Cg)^{x^\varphi}=Cg$, and so $gxg^{-1}\in C$, completing the proof.
\end{proof}

The following notation will be used throughout the discussion related to almost simple groups with socle $\PSL_2(q)$.

\begin{notation}\label{nota:PGaL}
Let $q=p^f\geq 4$ for some prime $p$ and positive integer $f$, let $\psi$ be the quotient from $\GaL_2(q)$ to $\PGaL_2(q)$ modulo scalars, let $\delta\in\Out(\PSL_2(q))$ be the diagonal automorphism, and let $\phi$ be the Frobenius field automorphism taking the $p$-th power. By abuse of notation, the automorphisms of $\GL_2(q)$ and $\PGL_2(q)$ respectively induced by $\phi$ are still denoted by $\phi$.
\end{notation}

The following result is well known, but we provide a proof here for self-containment.

\begin{lemma}\label{lm:non-split}
  Under Notation~$\ref{nota:PGaL}$, if $q$ is odd and $f$ is even, then the extension $\PSL_2(q).\la \delta\phi^{f/2}\ra$ is non-split.
\end{lemma}

\begin{proof}
  Write $r=p^{f/2}$. It suffices to prove that for each $g\in \GL_2(q)$ with $\det(g)$ a non-square in $\FF_q^\times$ we have $(g\phi^{f/2})^2\notin \ZZ(\GL_2(q))$, or equivalently, there is no $\lambda\in \FF_q^\times$ such that $g\overline{g}=\lambda I_2$, where $\overline{g}:=g^{\phi^{f/2}}$. Suppose for a contradiction that
  \[
    g\overline{g}=\lambda I_2 \ \text{for some}\ g\in \GL_2(q) \ \text{with}\ \det(g)\ \text{a non-square and}\ \lambda\in \FF_q^\times.
  \]
  It follows that $\overline{g}=\lambda g^{-1}$ commutes with $g$, and so $\lambda I_2=g\overline{g}=\overline{g}g=\overline{g\overline{g}}=\overline{\lambda I_2}=\lambda^rI_2$. Thus,
  \begin{equation}\label{eq:lambdar-1}
    \lambda^{r-1}=1.
  \end{equation}
  If the Jordan canonical form of $g$ is 
  $
  \begin{pmatrix}
    \mu & 1\\
    0 & \mu
  \end{pmatrix}
  $ 
  for some $\mu\in \overline{\FF_q}$, then it follows from $2\mu=\mathrm{Tr}(g)\in \FF_q$ that $\mu \in \FF_q$, and so $\det(g)=\mu^2$ contradicts the condition that $\det(g)$ is non-square in $\FF_q^\times$. Therefore, the Jordan canonical form of $g$ is 
  $
  \begin{pmatrix}
    \lambda_1 & 0\\
    0 & \lambda_2
  \end{pmatrix}
  $
  for some $\lambda_1,\lambda_2\in \overline{\FF_q}$. Then both $g$ and $\overline{g}$ are diagonalizable. This combined with $g\overline{g}=\overline{g}g$ implies that there exists $h\in \GL_2(\overline{\FF_q})$ with
  \[
    h^{-1}gh=\begin{pmatrix}
      \lambda_1 & 0\\
    0 & \lambda_2
    \end{pmatrix} \ \text{ and }\
    h^{-1}\overline{g}h
    =\overline{\begin{pmatrix}
      \lambda_1 & 0\\
    0 & \lambda_2
    \end{pmatrix}}
    =\begin{pmatrix}
      \lambda_1^r & 0\\
    0 & \lambda_2^r
    \end{pmatrix}.
  \]
  Then we derive from $g\overline{g}=\lambda I_2$ that
  \[
    \lambda I_2=h^{-1}\lambda I_2 h=(h^{-1}gh)(h^{-1}\overline{g}h)=\begin{pmatrix}
        \lambda_1^{r+1} & 0\\
        0 & \lambda_2^{r+1}
    \end{pmatrix},
  \]
  and so $\lambda_1^{r+1}=\lambda=\lambda_2^{r+1}$. This together with~\eqref{eq:lambdar-1} leads to
  \[
    \det(g)^{(q-1)/2}=(\lambda_1\lambda_2)^{(q-1)/2}=(\lambda_1^{r+1}\lambda_2^{r+1})^{(r-1)/2}=(\lambda^2)^{(r-1)/2}=1,
  \]
  which implies that $\det(g)$ is a square in $\FF_q^\times$, a contradiction.
\end{proof}

Recall that $\PGL_2(q)$ is a permutation group on the set of one-dimensional subspaces of $\FF_q^2$ with the point stabilizer $C_p^f\rtimes C_{q-1}$. Moreover, every maximal subgroup of $\PGL_2(q)$ of the form  $C_p^f\rtimes C_{q-1}$ is a point stabilizer.

\begin{lemma}\label{lm:PGL2q}
    Under Notation~$\ref{nota:PGaL}$, let $q\equiv 3\pmod{4}$, let $M$ be a maximal subgroup of $G:=\PGL_2(q)$ isomorphic to $C_p^f\rtimes C_{q-1}$, and let $Q\in \Syl_2(M)$. Then any Sylow $2$-subgroup of $\NN_G(Q)$ is isomorphic to $D_4$.
\end{lemma}

\begin{proof}
    Without loss of generality, let $M=G_{\la(0,1)\ra}$, where $\la (0,1)\ra$ is the $1$-dimensional subspace of $\FF_q^2$ generated by the vector $(0,1)$, and let $\FF_q^\times=\la \zeta\ra$. Then,
    \[
        M=\left\{\begin{pmatrix}
            a & b\\
            0 & d
        \end{pmatrix}^\psi \Big|~ a,b,d\in \FF_q,\ ad\neq 0\right\}
        =\Big\la \begin{pmatrix}
            1 & 1\\
            0 & 1
        \end{pmatrix}^\psi \Big\ra \rtimes \Big\la \begin{pmatrix}
            1 & 0\\
            0 & \zeta
        \end{pmatrix}^\psi \Big\ra  \cong C_p^f\rtimes C_{q-1}.
    \]
    Since $Q\in \Syl_2(M)$ and $q\equiv 3\pmod{4}$, we may assume without loss of generality that
    \[
        Q=\Big\la \begin{pmatrix}
            1 & 0\\
            0 & -1
        \end{pmatrix}^\psi \Big\ra\cong C_2.
    \]
    For each $x:=\begin{pmatrix}
        a & b\\
        c & d
    \end{pmatrix}^\psi\in G$, where $ad-bc\neq 0$, the element $x \in \NN_G(Q)$ if and only if
    \[
        \begin{pmatrix}
            a & b\\
            c & d
        \end{pmatrix}^\psi \begin{pmatrix}
            1 & 0\\
            0 & -1
        \end{pmatrix}^\psi=\begin{pmatrix}
            1 & 0\\
            0 & -1
        \end{pmatrix}^\psi \begin{pmatrix}
            a & b\\
            c & d
        \end{pmatrix}^{\psi}.
    \]
    This equality holds if and only if there exists $\lambda\in \FF_q^\times$ such that
    \[
        \begin{pmatrix}
            a & -b\\
            c & -d
        \end{pmatrix}=
        \lambda\begin{pmatrix}
            a & b\\
            -c & -d
        \end{pmatrix},
    \]
    which is equivalent to $ac=bd=0$. It follows that
    \[
        \NN_G(Q)=\Big\{\begin{pmatrix}
            1 & 0\\
            0 & d
        \end{pmatrix}^\psi \Big|~ d\in \FF_q^\times\Big\}\cup \Big\{\begin{pmatrix}
            0 & 1\\
            c & 0
        \end{pmatrix}^\psi \Big|~ c\in \FF_q^\times\Big\}.
    \]
    Therefore, $|\NN_G(Q)|=2(q-1)$, and so any Sylow $2$-subgroup of $\NN_G(Q)$ has order $2(q-1)_2=4$ and is conjugate to
    \[
        \Big\la \begin{pmatrix}
            1 & 0\\
            0 & -1
        \end{pmatrix}^\psi,\begin{pmatrix}
            0 & 1\\
            -1 & 0
        \end{pmatrix}^\psi\Big\ra\cong D_4.
    \]
    This completes the proof.
\end{proof}

We end this section by presenting the following two technical lemmas needed in Section~\ref{sec:6}.

\begin{lemma}\label{lm:C8Q8}
    Under Notation~$\ref{nota:PGaL}$, let $q\equiv 1\pmod{8}$, let $f$ be even, let $G=T\rtimes\la \phi^{f/2}\ra$, and let $H$ be a subgroup of $G$ such that $H\cong C_4$ and $H\cap T\cong C_2$. Then there does not exist $K\leq G$ such that $H\leq K$ and  $K\cong C_8$ or $Q_8$.
\end{lemma}

\begin{proof}
    Let $\FF_q^\times=\la \zeta\ra$, $r=p^{f/2}$, $n=(q-1)_2$, $n'=(q-1)_{2'}$,
    \[
        x=\begin{pmatrix}
            \zeta^{n'} & 0\\
            0 & \zeta^{-n'}
        \end{pmatrix}^\psi \ \text{ and }\
        y=\begin{pmatrix}
            0 & 1\\
            -1 & 0
        \end{pmatrix}^\psi.
    \]
    Then it is straightforward to verify that 
    \[
        |x|=\frac{n}{2}, \ \ |y|=2,\ \ |\phi|=f,\ \ x^y=x^{-1}, \ \ x^\phi=x^p,\ \text{ and }\ y^\phi=y.
    \]
    Let $Q=\la x,y\ra\rtimes\la \phi^{f/2}\ra$. As $\la x,y\ra \cong D_n$, we have $\la x,y\ra\in \Syl_2(T)$ and $Q\in \Syl_2(G)$. 
    
    Suppose for a contradiction that there exists $K$ such that $H\leq K\leq G$ and  $K\cong C_8$ or $Q_8$. It follows that there exists $g\in G$ such that $H^g\leq K^g\leq Q$. Moreover, $H^g\cong H\cong C_4$ and $H^g\cap T=(H\cap T)^g\cong C_2$. Write $H^g=\la z\ra$ for some $z\in G$. Since $H^g\cap T<H^g$, we have $z\notin T$.

    If $K\cong C_8$, then $K^g=\la w\ra\cong C_8$ for some $w\in G$, and so $z\in\la w^2\ra\leq T$, contradicting that $z\notin T$. Therefore, $K\cong Q_8$. Write
    \[
        K^g=\la z, u\ra\cong Q_8 \ \text{ for some }\ u\in Q \ \text{ with }\ u^2=z^2 \ \text{ and }\ z^u=z^{-1}.
    \]
    Since $z,u\in Q=\la x,y\ra\rtimes\la\phi^{f/2}\ra$ and $z\notin T$, we have 
    \[
        z=x^iy^j\phi^{f/2} \ \text{ and }\ u=x^ky^\ell (\phi^{f/2})^s \ \text{ for some }\ i,k\in [n/2], \ \text{and}\  j,\ell,s\in [2].
    \]
    Note that $z^u=z^{-1}$ is equivalent to $zu=uz^{-1}$. By a straightforward calculation,
    \begin{align*}
        zu
        =x^iy^j\phi^{f/2}x^ky^\ell (\phi^{f/2})^s
        &=x^iy^j(x^ky^\ell)^{\phi^{f/2}}(\phi^{f/2})^{s+1}\\
        &=x^iy^jx^{kr}y^\ell(\phi^{f/2})^{s+1}
        =x^{i+(-1)^jkr}y^{j+\ell}(\phi^{f/2})^{s+1};
    \end{align*} 
    \begin{align*}
        uz^{-1}
        =x^ky^\ell (\phi^{f/2})^s\phi^{f/2}y^jx^{-i}
        &=x^ky^\ell(y^jx^{-i})^{(\phi^{f/2})^{s+1}}(\phi^{f/2})^{s+1}\\
        &=x^ky^\ell y^jx^{-ir^{s+1}}(\phi^{f/2})^{s+1}
        =x^{k-(-1)^{\ell+j}ir^{s+1}}y^{\ell+j}(\phi^{f/2})^{s+1}.
    \end{align*}
    These combined with $zu=uz^{-1}$ give that $i+(-1)^jkr\equiv k-(-1)^{\ell+j}ir^{s+1} \pmod{n/2}$, and so
    \begin{equation}\label{eq:ikmod}
        i(1+(-1)^{\ell+j}r^{s+1})\equiv k(1-(-1)^jr)\pmod{n/2}.
    \end{equation}
    A straightforward calculation gives that 
    \[
        z^2=(x^iy^j\phi^{f/2})^2=x^iy^j(x^iy^j)^{\phi^{f/2}}=x^iy^jx^{ir}y^j=x^{i+(-1)^jir}=x^{i(1+(-1)^jr)}.
    \]
    Similarly, $u^2=x^{k(1+(-1)^\ell r^s)}$.
    These together with $|z|=4=|u|$ imply that 
    \begin{equation}\label{eq:ik}
        i(1+(-1)^jr)\equiv \frac{n}{4} \equiv k(1+(-1)^\ell r^s) \pmod{ n/2 }.
    \end{equation}
    Noting that $n=(r+1)_2(r-1)_2=|(-1)^jr+1|_2|(-1)^jr-1|_2$, we obtain from~\eqref{eq:ik} that
    \begin{equation}\label{eq:irj}
        i_2= \frac{|(-1)^jr-1|_2}{4},
    \end{equation}
    and so $(-1)^jr\equiv 1\pmod{4}$.

    Assume first that $s=1$. Along the same line as the above paragraph, we derive from~\eqref{eq:ik} that $(-1)^\ell r\equiv 1\pmod{4}$, which together with $(-1)^jr\equiv 1\pmod{4}$ gives $(-1)^j=(-1)^\ell$. Then it follows from~\eqref{eq:ikmod} that $2i\equiv k(1-(-1)^jr)\pmod{n/2}$. This implies that $(1-(-1)^jr)_2$ divides $2i$, contradicting~\eqref{eq:irj}.

   Assume next that $s=0$. Then it follows from~\eqref{eq:ik} that $\ell=0$ and $k_2=n/8$. This together with~\eqref{eq:ik} and~\eqref{eq:ikmod} gives that 
   \[
     \frac{n}{4}\equiv i_2(1+(-1)^jr)_2\equiv k_2(1-(-1)^jr)_2= \frac{n}{4}\cdot\frac{(1-(-1)^jr)_2}{2}\pmod{n/2},
   \]
   which forces $(-1)^jr\equiv -1\pmod{4}$, contradicting the conclusion $(-1)^jr\equiv 1\pmod{4}$.
\end{proof}

\begin{lemma}\label{lm:Q8}
    Under Notation~$\ref{nota:PGaL}$, let $q\equiv 1\pmod{8}$, let $f$ be even, let $G=T.\la \delta\phi^k\ra$ for some integer $k$ such that $|\la \phi^k\ra|_2=2$, and let $H$ be a subgroup of $G$ such that $H\cong C_4$ and $H\cap T\cong C_2$. Then there exists $K\leq G$ such that $H\leq K$ and $K\cong Q_8$.
\end{lemma}

\begin{proof}
    Construct a Sylow $2$-subgroup of $G$ as follows. Let $\FF_q^\times=\la \zeta\ra$, $n=(q-1)_2$, $n'=(q-1)_{2'}$, $r=p^{f/2}$,
    \[
        x=\begin{pmatrix}
            \zeta^{n'} & 0\\
            0 & \zeta^{-n'}
        \end{pmatrix}^\psi, \ \ 
        y=\begin{pmatrix}
            0 & 1\\
            -1 & 0
        \end{pmatrix}^\psi \ \text{ and }\
        z=\begin{pmatrix}
            \zeta^{n'} & 0\\
            0 & 1
        \end{pmatrix}^\psi.
    \]
    A straightforward calculation gives that 
    \begin{equation}\label{eq:xyz}
        |x|=\frac{n}{2}, \ \ |y|=2,\ \ x^y=x^{-1}, \ \ x^{z\phi^{f/2}}=x^r,\ \text{ and }\ y^{z\phi^{f/2}}=(x^{-1}y)^{\phi^{f/2}}=x^{-r}y.
    \end{equation}
    Moreover, we have 
    \begin{equation}\label{eq:zphi}
        (z\phi^{f/2})^2=zz^{\phi^{f/2}}=z^{1+r}=\Bigg((\zeta^{n'})^{\frac{1+r}{2}}
        \begin{pmatrix}
            (\zeta^{n'})^{\frac{1+r}{2}} & 0\\
            0 & (\zeta^{n'})^{-\frac{1+r}{2}}
        \end{pmatrix}\Bigg)^\psi=x^{\frac{1+r}{2}}.
    \end{equation}
    Let $Q=\la x,y\ra\la z\phi^{f/2}\ra$. Then $Q=\la x,y\ra.\la \delta\phi^{f/2}\ra\cong D_n.C_2$, and so $|Q|=|G|_2$. Noting from $|\la \phi^k\ra|_2=2$ that $\delta\phi^{f/2}\in \la \delta\phi^k \ra$, we have $Q\leq G$, and so $Q\in \Syl_2(G)$. 
    
    Since $H\cong C_4$, there exists $g\in G$ such that $H^g\leq Q$. Then $H^g\cap T=(H\cap T)^g\cong C_2$. Write $H^g=\la w\ra$. Then $|w|=4$ and $w\notin Q\cap T=\la x,y\ra$. Hence, 
    \[
        w=x^iy^jz\phi^{f/2} \ \text{ for some }\ i\in [n/2] \ \text{ and }\ j\in [2]. 
    \]
    In view of~\eqref{eq:xyz} and~\eqref{eq:zphi}, it follows that 
    \[
        w^2
        =(x^iy^jz\phi^{f/2})^2
        =x^iy^j(z\phi^{f/2})^2(x^iy^j)^{z\phi^{f/2}}
        =x^iy^jx^{\frac{1+r}{2}}x^{ri}(x^{-r}y)^j
        =x^{i+(-1)^j(\frac{1+r}{2}+r(i-j))}.
    \]
    This combined with $|w|=4$ implies that 
    \begin{equation}\label{eq:ijr}
        \Big|i+(-1)^j\Big(\frac{1+r}{2}+r(i-j)\Big)\Big|_2=\frac{n}{4}.
    \end{equation}
    Moreover, it holds for both $j=0$ and $j=1$  that
    \[
        \Big|i+(-1)^j\Big(\frac{1+r}{2}+r(i-j)\Big)\Big|_2
        =\frac{|1+(-1)^jr|_2}{2}.
    \]
    This combined with~\eqref{eq:ijr} yields 
    \begin{equation}\label{eq:jrmod4}
       |1+(-1)^jr|_2=\frac{n}{2}.
    \end{equation}

    Let $u=x^{n/8}$. We complete the proof by showing that $H^g\leq \la u,w\ra\cong Q_8$. Since $|u|=4=|w|$, we are left to show that $u^w=u^{-1}$. In fact, by~\eqref{eq:xyz} and~\eqref{eq:jrmod4},
    \[
        u^w=(x^{\frac{n}{8}})^{x^iy^jz\phi^{f/2}}=x^{\frac{n}{8}(-1)^jr}=x^{-\frac{n}{8}}x^{\frac{n}{8}(1+(-1)^jr)}=x^{-\frac{n}{8}}=u^{-1},
    \]
    as desired.
\end{proof}



\section{A new criterion for subgroup perfect codes}\label{sec:3}

In this section, we provide a new criterion for determining when a subgroup is a perfect code. Applying this criterion, we then derive necessary and sufficient conditions for a subgroup with certain structures to be a perfect code.

\begin{theorem}\label{thm:pciff}
    Let $G$ be a group, and let $H$ be a $2$-subgroup of $G$. Then $H$ is a perfect code of $G$ if and only if, for each $a\in \NN_G(H)\setminus H$ with $a^2\in H$, the subgroup $H$ has a complement in $H\la a\ra$.
\end{theorem}

\begin{proof}
    Assume that $H$ is a perfect code of $G$. Take an arbitrary $a\in \NN_G(H)\setminus H$ with $a^2\in H$. It follows that $|H|/|H\cap H^a|=1$ is odd. Thus, we derive from Theorem~\ref{thm:1} that there exists $b\in aH=Ha$ such that $b^2=e$, which means that $\la b\ra$ is a complement of $H$ in $H\la a\ra$.

    Conversely, assume that, for each $a\in \NN_G(H)\setminus H$ with $a^2\in H$, the subgroup $H$ has a complement in $H\la a\ra$. Take an arbitrary $a\in G$ such that $a^2\in H$ and $|H|/|H\cap H^a|$ is odd. Then, since $H$ is a $2$-group, we deduce that $|H|/|H\cap H^a|=1$, which means $a\in \NN_G(H)$. If $a\in H$, then $aH=H$ obviously contains an involution. If $a\in \NN_G(H)\setminus H$, then our assumption gives that $H$ has a complement, say, $\la b\ra$ in $H\la a\ra$. In this case, $b\in aH$ and $b$ is an involution. This completes the proof by Theorem~\ref{thm:1}.
\end{proof}

The above theorem reveals that whether a subgroup $H$ is a perfect code of a group $G$ depends on whether $H$ admits ``local complements" of order $2$. This suggests that the possible structures of $H.C_2$ play an important role in determining subgroup perfect codes.

Note that as Lemma~\ref{lm:Z2023} shows, to determine the subgroup perfect codes of a group, it suffices to determine which $2$-subgroups can occur as perfect codes. The following lemma provides a necessary and sufficient condition for a cyclic $2$-group to be a perfect code of any overgroup $G$.

\begin{lemma}\label{lm:cyclic}
    Let $G$ be a group, and let $H$ be a non-trivial cyclic $2$-subgroup of $G$. Then $H$ is a perfect code of $G$ if and only if $H$ is not contained in any subgroup of $K$ of $G$ such that $|K|/|H|=2$ and $K$ is either cyclic or generalized quaternion.
\end{lemma}

\begin{proof}
    Take an arbitrary $a\in \NN_G(H)\setminus H$ with $a^2\in H$. It is clear that $|H\la a\ra|/|H|=2$, and so $H\la a\ra$ is a $2$-group. Since $H$ contains a unique involution, it follows that $H$ has a complement in $H\la a\ra$ if and only if $H\la a\ra$ contains at least two involutions, which, by Lemma~\ref{lm:2group}, is equivalent to $H\la a\ra$ being neither cyclic nor generalized quaternion. This together with Theorem~\ref{thm:pciff} completes the proof.
\end{proof}

Using similar arguments, we give a necessary and sufficient condition for a generalized quaternion $2$-subgroup to be a perfect code as follows.

\begin{lemma}\label{lm:quaternion}
    Let $G$ be a group, and let $H$ be a generalized quaternion $2$-subgroup of $G$. Then $H$ is a perfect code of $G$
    if and only if $H$ is a maximal generalized quaternion subgroup of $G$.
\end{lemma}

\begin{proof}
    Take an arbitrary $a\in \NN_G(H)\setminus H$ with $a^2\in H$. Since $H$ is generalized quaternion, $H$ has a unique involution. Thus, $H$ has a complement in $H\la a\ra$ if and only if $H\la a\ra$ contains at least two involutions, which, by Lemma~\ref{lm:2group}, is equivalent to $H$ being a maximal generalized quaternion subgroup of $G$. This combined with Theorem~\ref{thm:pciff} completes the proof.
\end{proof}

Let $H$ be a group of the form $C_{2^n}\rtimes C_2$, where $n\geq 2$ is an integer. Then there exist $x,y\in H$ such that $|x|=2^n$, $|y|=2$ and $H=\la x\ra \rtimes \la y\ra$. It follows that $x^y=x^k$ for some $k\in \Z$. We deduce from  $|y|=2$ that $x=x^{y^2}=(x^k)^y=x^{k^2}$. Thus, $k^2\equiv 1\pmod{2^n}$, and so
    \begin{equation}\label{eq:k}
        x^y=x^k\ \text{ with }\ k\equiv \pm 1\text{ or}\, \pm 1+2^{n-1} \pmod{2^n}.
    \end{equation}

\begin{lemma}\label{lm:C2nrtimesC2}
    Let $G$ be a group, let $H\leq G$ such that $H\cong C_{2^n}\rtimes C_2$ for some integer $n\geq 2$, and let $a\in \NN_G(H)\setminus H$ with $a^2\in H$. Write $H=\la x\ra \rtimes \la y\ra$ with elements $x$ and $y$ of orders $2^n$ and $2$ respectively such that~\eqref{eq:k} holds. Then $H$ has a complement in $H\la a\ra$ if and only if all the following conditions hold:
    \begin{enumerate}[\rm(a)]
        \item $a^2\in \la x\ra$;
        \item if $\la x, a\ra\cong C_{2^{n+1}}$, then $H\cong D_{2^{n+1}}$ and $a^y=a^{-1}$;
        \item \label{enu:cH} if $H\cong C_{2^n}\times C_2$ and $\la x, a\ra\cong Q_{2^{n+1}}$, then $[a,y]=a^2$;
        \item \label{enu:dH} if $H\cong D_{2^{n+1}}$ and $\la x, a\ra\cong Q_{2^{n+1}}$, then $[a,y]\in \la x^2 \ra$;
    \end{enumerate}
\end{lemma}

\begin{proof}
    Let $K=H\la a\ra$ and $Q=\la x,a\ra$. We examine whether $H$ has a complement in $K$ in each of the following cases, from which the conclusion of the lemma follows.

    \textsf{Case 1:} $a^2\notin \la x\ra$. Then $a^2\in H\setminus\la x\ra =\la x\ra y$, and so $\la x\ra \la a\ra=H\la a\ra$. We prove that $H$ does not have a complement in $K$ in this case. Suppose for a contradiction that $H$ has a complement in $K$. Then $K=H\rtimes\la z\ra$ for some involution $z\in K$ since $|K|=2 |H|$. It follows that
    \[
        (\la x\ra \la a\ra)/\la x\ra= (H\la a\ra)/\la x\ra=(H\rtimes\la z\ra)/\la x\ra=(H/\la x\ra)\rtimes(\la z\ra \la x\ra /\la x\ra)\cong C_2\rtimes C_2=C_2^2.
    \]
    However, $ (\la x\ra \la a\ra)/\la x\ra$ is cyclic, a contradiction.

    \textsf{Case 2:} $a^2 \in \la x\ra$. In this case, $Q=\la x\ra \la a\ra$, $|K|=2|H|=2|Q|$, and $K=Q\la y \ra$. By Lemma~\ref{lm:2group}, there are precisely three subcases as follows.

    \textsf{Subcase 2.1:} $Q$ contains at least two involutions. Since $\la x\ra$ contains a unique involution, there exists an involution $w\in \la x\ra a\subseteq Ha$. Then we conclude from $|K|=2|H|$ that $\la w\ra$ is a complement of $H$ in $K$.

    \textsf{Subcase 2.2:} $Q$ is cyclic. Then $Q=\la x,a\ra=\la a\ra$ and  $|a|=2^{n+1}$. Write
    \[
        x=a^s\ \text{ for some even }\, s\in \Z.
    \]
    Since $y$ normalizes $Q$, we have
    \[
        a^y=a^\ell\ \text{ for some }\, \ell \in \Z.
    \]
    It follows from $y^2=e$ that $a=a^{y^2}=a^{\ell^2}$, which implies that $\ell^2\equiv 1\pmod{2^{n+1}}$, and so
    \[
        \ell\equiv \pm1\text{ or } \pm1+2^n \pmod{2^{n+1}}.
    \]
    Recall from~\eqref{eq:k} that $x^y=x^k$ for some $k\equiv \pm 1\text{ or}\, \pm 1+2^{n-1} \pmod{2^n}$. Since
    \[
        x^k=x^y=(a^s)^y=a^{\ell s}=x^\ell,
    \]
    we obtain that $k\equiv \ell \pmod{2^n}$, which leads to
    \[
        \begin{cases}
            k\equiv 1\pmod{2^n}\\
            \ell\equiv 1\text{ or } 1+2^n\pmod{2^{n+1}}
        \end{cases}
        \text{ or }\
        \begin{cases}
            k\equiv -1\pmod{2^n}\\
            \ell\equiv -1\text{ or } -1+2^n\pmod{2^{n+1}}.
        \end{cases}
    \]
    Recalling that $s$ is even, we have
    \[
        (x^ia)^2=x^i a x^i a=x^{2i}a^2=a^{2+2is}=a^{2(1+is)}\neq 1.
    \]
    Moreover, since
    \[
        (x^iya)^2=x^iyax^iya=x^ia^y(x^y)^ia=x^ia^\ell x^{ki} a=a^{s(k+1)i+\ell+1},
    \]
    we derive that $(x^iya)^2=e$ for some $i$ if and only if $k\equiv -1\pmod{2^n}$ and $\ell\equiv -1\pmod{2^{n+1}}$. Therefore, in this subcase, $H$ has a complement in $K$ if and only if $x^y=x^{-1}$ and $a^y=a^{-1}$.

    \textsf{Subcase 2.3:} $Q$ is generalized quaternion. Then since $|Q|=|\la x\ra \la a\ra|=2^{n+1}$, we have $Q\cong Q_{2^{n+1}}$. Hence, from the structure of $Q_{2^{n+1}}$ we know that $\langle x\rangle$ is the unique cyclic subgroup of $Q$ of order $2^n$, and
    \[
        x^a=x^{-1}\ \text{ and }\ a^2=x^{2^{n-1}}.
    \]
    Since $a$ normalizes $H=\la x\ra \cup \la x\ra y$, we have
    \[
        y^a=x^ty \ \text{ for some }\, t\in \Z.
    \]
    It follows from $y^2=e$ that
    \[
    e=(y^2)^a=(y^a)^2=(x^t y)^2=x^t (x^t)^y=x^{t+tk}=x^{t(k+1)},
    \]
    which implies
    \begin{equation}\label{eq:tk}
        t(k+1)\equiv 0\pmod{2^n}.
    \end{equation}
    Note that, for each $i\in \Z$,
    \begin{align*}
        &(x^ia)^2=x^iax^ia=x^ia^2(x^i)^a=x^ix^{2^{n-1}}x^{-i}=x^{2^{n-1}}\neq e,\\
        &
        (x^iya)^2
        =x^iya^2(x^iy)^a
        =x^iyx^{2^{n-1}}x^{-i}x^ty
        =x^i(x^{2^{n-1}-i+t})^y
        =x^{i+(2^{n-1}-i+t)k}
        =x^{(1-k)i+(2^{n-1}+t)k}.
    \end{align*}
    Thus, $Ha$ contains an involution if and only if
    \begin{equation}\label{eq:ki}
        \text{there exists $i\in \Z$ such that } (1-k)i+(2^{n-1}+t)k\equiv 0\pmod{2^n}.
    \end{equation}
    As $k\equiv \pm 1\pmod{4}$ when $n=2$ and $k\equiv \pm 1\text{ or}\, \pm 1+2^{n-1} \pmod{2^n}$ when $n\geq 3$, we obtain the consequence of~\eqref{eq:tk} on $t$ in the following Tables~\ref{tab:kit2} and~\ref{tab:kit} for each possibility of $k$. Then the equivalent condition for~\eqref{eq:ki} can be derived in each row of the tables.
\begin{table}[h]
    \centering
    \renewcommand{\arraystretch}{1.5}
    \begin{tabular}{c|c|c}
    \noalign{\hrule height 1.2pt}
    Possibility for $k$ & Consequence of~\eqref{eq:tk}  & Equivalent condition for~\eqref{eq:ki} \\
    \noalign{\hrule height 1.2pt}
    $k\equiv 1\pmod{4}$ & $t\text{ even }$ & $t\equiv 2\pmod{4}$\\
    \hline
    $k\equiv-1\pmod{4}$ & -- & $t$ even \\
    \noalign{\hrule height 1.2pt}
    \end{tabular}
    \caption{$k$ and $t$ for Subcase~2.3 in the proof of Lemma~\ref{lm:C2nrtimesC2} when $n=2$.}
    \label{tab:kit2}
\end{table}
\begin{table}[h]
    \centering
    \renewcommand{\arraystretch}{1.5}
    \begin{tabular}{c|c|c}
    \noalign{\hrule height 1.2pt}
    Possibility for $k$ & Consequence of~\eqref{eq:tk}  & Equivalent condition for~\eqref{eq:ki} \\
    \noalign{\hrule height 1.2pt}
    $k\equiv 1\pmod{2^n}$ & $t\equiv 0\text{ or } 2^{n-1}\pmod{2^n}$ & $t\equiv 2^{n-1}\pmod{2^n}$\\
    \hline
    $k\equiv 1+2^{n-1}\pmod{2^n}$ & $t\equiv 0\text{ or } 2^{n-1}\pmod{2^n}$ & $t\equiv 0\text{ or } 2^{n-1}\pmod{2^n}$ \\
    \hline
    $k\equiv-1\pmod{2^n}$ & -- & $t$ even \\
    \hline
    $k\equiv -1+2^{n-1}\pmod{2^n}$ & $t$  even & $t\text{ even}$ \\
    \noalign{\hrule height 1.2pt}
    \end{tabular}
    \caption{$k$ and $t$ for Subcase~2.3 in the proof of Lemma~\ref{lm:C2nrtimesC2} when $n\geq 3$.}
    \label{tab:kit}
\end{table}
    For example, when $n\geq 3$ and $k\equiv -1+2^{n-1}\pmod{2^n}$,~\eqref{eq:tk} implies that $t$ is even, and so~\eqref{eq:ki} is equivalent to saying that there exists $i\in \Z$ such that
    \begin{equation}\label{eq:2t}
        (2-2^{n-1})i+(2^{n-1}+t)(-1+2^{n-1}) \equiv 0\pmod{2^n}.
    \end{equation}
    Given the condition that $t$ is even,~\eqref{eq:2t} is equivalent to
    \[
        (1-2^{n-2})i-2^{n-2}-\frac{t}{2} \equiv 0\pmod{2^{n-1}},
    \]
    which always has a solution $i\in \Z$. These give the equivalent condition for~\eqref{eq:ki} in the last row of Table~\ref{tab:kit}.

    Noting that $[a,y]=(y^{-1})^ay=y^ay=(x^ty)y=x^t$, we conclude that $t$ is even if and only if $[a,y]\in \la x^2\ra$, that $t\equiv 0\pmod{2^n}$ if and only if $[a,y]=e$, and that $t\equiv 2^{n-1}\pmod{2^n}$ if and only if $[a,y]=a^2$. Hence, in this subcase, $H$ has a complement in $K$ if and only if conditions~\eqref{enu:cH} and~\eqref{enu:dH} hold.
\end{proof}

The following two corollaries are direct consequences of Theorem~\ref{thm:pciff} and Lemma~\ref{lm:C2nrtimesC2}.

\begin{corollary}\label{cor:Da^2inQ}
    Let $n\geq 2$ be an integer, and let $G$ be a group with a subgroup
    \[
        H=\langle x, y \mid x^{2^n}=e,\, y^2=e,\, x^y=x^{-1}\rangle\cong D_{2^{n+1}}.
    \]
    Then $H$ is a perfect code of $G$ if and only if for each $a\in \NN_G(H)\setminus H$ with $a^2\in H$, all the following conditions hold:
    \begin{enumerate}[\rm(a)]
        \item $a^2\in \langle x\rangle$;
        \item if $\la x, a\ra\cong C_{2^{n+1}}$, then $a^y=a^{-1}$;
        \item if $\la x, a\ra\cong Q_{2^{n+1}}$, then $[a,y]\in \la x^2 \ra$.
    \end{enumerate}
\end{corollary}

\begin{corollary}\label{cor:a^2inQ}
    Let $n\geq 3$ be an integer, and let $G$ be a group with a subgroup
    \[
        H=\langle x, y \mid x^{2^n}=e,\, y^2=e,\, x^y=x^{-1+2^{n-1}}\rangle.
    \]
    Then $H$ is a perfect code of $G$ if and only if $a^2\in \langle x\rangle$ for each $a\in \NN_G(H)\setminus H$ with $a^2\in H$.
\end{corollary}

For a $2$-subgroup $H$ of a group $G$, the structure of $H$ is constrained by the structure of the Sylow $2$-subgroups of $G$. In light of the exposition in Lemma~\ref{lm:C2nrtimesC2}, we obtain the follow criterion for determining whether $H$ is a perfect code of $G$ when $G$ has dihedral Sylow $2$-subgroups.

\begin{corollary}\label{cor:dihedral}
    Let $G$ be group whose Sylow $2$-subgroups are isomorphic to $D_{2n}$ for some integer $n\geq 2$, and let $H$ be a $2$-subgroup of $G$. Then $H$ is not a perfect code of $G$ if and only if there exists a cyclic $2$-subgroup $C$ of $G$ such that $1<H<C$.
\end{corollary}

\begin{proof}
    If $H=1$, then it follows form Theorem~\ref{thm:1} that $H$ is a perfect code of $G$. In the following, assume that $H>1$. Since $H$ is a $2$-subgroup of $G$, it is contained in some Sylow $2$-subgroup of $G$, and so $H$ is either dihedral or cyclic.

    Assume first that $H$ is dihedral. Then there exist $x,y\in H$ such that $H=\la x,y\ra$ with $x^m=y^2=e$ and $x^y=x^{-1}$. Take an arbitrary $a\in \NN_G(H)\setminus H$ with $a^2\in H$. Since $H\la a\ra$ is a $2$-group, it follows that $H\la a\ra$ is contained in some Sylow $2$-subgroup of $G$, and so $H\la a\ra$ is also dihedral. This implies that one of the following holds:
    \begin{enumerate}[\rm(i)]
        \item \label{enu:iinvo} $a$ is an involution;
        \item \label{enu:iiinvo} $\la x\ra$ is an index-$2$ subgroup of $\la a\ra$ and $H\la a\ra=\la a\ra\rtimes \la y\ra$.
    \end{enumerate}
    For case~\eqref{enu:iinvo}, $\la a\ra$ is a complement of $H$ in $H\la a\ra$. For case~\eqref{enu:iiinvo}, $a^2\in \la x\ra$ and $a^y=a^{-1}$, whence we conclude from Lemma~\ref{lm:C2nrtimesC2} that $H$ has a complement in $H\la a\ra$. Thus, by Theorem~\ref{thm:pciff}, $H$ is a perfect code of $G$.

    Assume next that $H$ is cyclic. Note that Sylow $2$-subgroups of $G$ are dihedral and hence do not contain generalized quaternion groups. Therefore, we conclude from Lemma~\ref{lm:cyclic} that $H$ is not a perfect code of $G$ if and only if there exists a cyclic $2$-subgroup $C\leq G$ such that $1<H<C$. This completes the proof.
\end{proof}

\section{Reduction via quotients}\label{sec:4}

We have established in Section~\ref{sec:2} a necessary condition for a subgroup $M$ to be a perfect code of a group $G$: with $K$ denoting $\mathrm{Core}_G(M)$, the subgroup $M/K$ is a perfect code of the quotient group $G/K$ (see Lemma~\ref{lm:quotient}\eqref{enu:quotient}). In this section, we study the lifting problem: assuming that $M/K$ is a perfect code of $G/K$, how can we determine whether this implies that $M$ is a perfect code of
$G$? To this end, we first present the following \emph{Diamond Lemma}, whose underlying idea is inspired by the Diamond Isomorphism Theorem (also known as the Second Isomorphism Theorem) in group theory, illustrated in Figure~\ref{fig:1}. Then we apply this lemma to show in Theorem~\ref{thm:split} that, whenever $G$ is a split extension of $K$ by $G/K$, the subgroup $M$ is a perfect code of $G$ if and only if $M/K$ is a perfect code of $G/K$.

\begin{lemma}[Diamond Lemma]\label{lm:diamond}
    Let $H$ and $K$ be subgroups of a group $G$ such that $HK$ is a group. If $H\cap K$ is a perfect code of $K$, then $H$ is a perfect code of $HK$.
\end{lemma}

\begin{figure}[h]
\begin{tikzpicture}[scale=0.6]
  \node (top) at (0,2) {\( HK \)};
  \node (left) at (-2,0) {\( H \)};
  \node (right) at (2,0) {\( K \)};
  \node (bottom) at (0,-2) {\( H \cap K \)};

  \draw (top) -- (left);
  \draw (top) -- (right);
  \draw (left) -- (bottom);
  \draw (right) -- (bottom);
\end{tikzpicture}
\caption{Diamond Lemma}
\label{fig:1}
\end{figure}
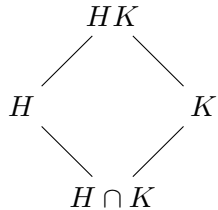

\begin{proof}
    Suppose that $H\cap K$ is a perfect code of $K$. It follows from Theorem~\ref{thm:1} that $H\cap K$ has an inverse-closed left transversal $L$ in $K$. Let $L=\{\ell_1,\ldots, \ell_t\}$, where $t=|K|/|H\cap K|$. Suppose that $\ell_i\ell_j^{-1}\in H$ for some $i,j\in [t]$. Then $\ell_i\ell_j^{-1}\in H\cap K$, and so $i=j$. This means that $\{\ell_1H,\ldots, \ell_tH\}$ are distinct left cosets of $H$ in $HK$. Noting that $t=|K|/|H\cap K|=|HK|/|H|$, we conclude that $L$ is a left transversal of $H$ in $HK$. Therefore, as $L$ is inverse-closed, $H$ is a perfect code of $HK$ by Theorem~\ref{thm:1}.
\end{proof}

The following example shows that the converse of Lemma~\ref{lm:diamond} does not hold, even if $H$ and $K$ are normal in $G$.

\begin{example}
    Let $m\geq 4$ be a $2$-power, and let
    \[
        G=\la x,y \mid x^m=y^2=e,\; xy=yx\ra\cong C_m\times C_2.
    \]
    Let $H=\la x\ra$ and $K=\la xy\ra$. Then $HK=G$ and $H\cap K=\la x^2\ra$. Hence, by Lemma~\ref{lm:cyclic}, $H$ is a perfect code of $HK$, but $H\cap K$ is not a perfect code of $K$. \qed
\end{example}

The next lemma shows that, however, if $|H|_2=|H\cap K|_2$, then the converse of Lemma~\ref{lm:diamond} holds.

\begin{lemma}\label{lm:diamondConverse}
    Let $H$ and $K$ be subgroups of a group $G$ such that $HK$ is a group and $|H|_2=|H\cap K|_2$. Then $H\cap K$ is a perfect code of $K$ if and only if $H$ is a perfect code of $HK$.
\end{lemma}

\begin{proof}
    By Lemma~\ref{lm:diamond}, we only need to show that if $H$ is a perfect code of $HK$, then $H\cap K$ is a perfect code $K$. Since $|H|_2=|H\cap K|_2$, we have $\Syl_2(H\cap K)\subseteq \Syl_2(H)$. Assume that $H$ is a perfect code of $HK$, and let $Q\in \Syl_2(H\cap K)$. Then $Q\in \Syl_2(H)$, which together with Lemma~\ref{lm:Z2023} implies that $Q$ is a perfect code of $HK$. Noting that $Q\leq K\leq HK$, we conclude from Lemma~\ref{lm:quotient}\eqref{enu:2.2a} that $Q$ is a perfect code of $K$. Thus, by Lemma~\ref{lm:Z2023}, $H\cap K$ is a perfect code of $K$, as desired.
\end{proof}

We now prove our main theorem of this section using the Diamond Lemma. This theorem shows that, in a split extension, determining whether $M$ is a perfect code of $G$ can be reduced to determining whether $M/K$ is a perfect code of $G/K$.

\begin{theorem}\label{thm:split}
    Let $G$ be a group, let $M$ be a subgroup of $G$, and let $K$ be a normal subgroup of $G$ contained in $M$. Suppose that $G$ is a split extension of $K$ by $G/K$. Then $M$ is a perfect code of $G$ if and only if $M/K$ is a perfect code of $G/K$.
\end{theorem}

\begin{proof}
    By Lemma~\ref{lm:quotient}\eqref{enu:quotient}, if $M$ is a perfect code of $G$, then $M/K$ is a perfect code of $G/K$. Conversely, suppose that $M/K$ is a perfect code of $G/K$. Since $G$ is a split extension of $K$ by $G/K$, the group $K$ has a complement $\tilde{G}$ in $G$ such that $\tilde{G}\cong G/K$. Let $\tilde{M}=M\cap \tilde{G}$. Then
    \[
        M=M\cap G=M\cap (K\tilde{G})=K(M\cap \tilde{G})=K\tilde{M}=K\rtimes \tilde{M},
    \]
    and so $\tilde{M}$ is the image of $M$ under the quotient modulo $K$. As $M/K$ is a perfect code of $G/K$, it follows that $\tilde{M}$ is a perfect code of $\tilde{G}$. Hence, by Lemma~\ref{lm:diamond}, $M$ is a perfect code of $M\tilde{G}=K\tilde{G}=G$, as required.
    \begin{figure}[h]
    \begin{tikzpicture}[scale=0.6]
      \node (top) at (0,2) {\( G=K\tilde{G} \)};
      \node (left) at (-2,0) {\( M \)};
      \node (right) at (2,0) {\( \tilde{G} \)};
      \node (bottom) at (0,-2) {\( \tilde{M}=M\cap \tilde{G} \)};

      \draw (top) -- (left);
      \draw (top) -- (right);
      \draw (left) -- (bottom);
      \draw (right) -- (bottom);
    \end{tikzpicture}
    \caption{Demonstration in the proof of Theorem~\ref{thm:split}}
    \label{fig:2}
    \end{figure}
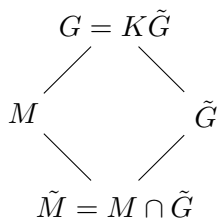
\end{proof}


We remark that the split extension condition in Theorem~\ref{thm:split} cannot be removed. This is shown by two examples in the rest of the section.

We adopt the notation in~\cite[\S 2.7]{W2009} to represent the double cover $2.S_n^-$ of the symmetric group $S_n$ for $n\geq 4$. In particular, the cycles in $2.S_n^-$ are written in square brackets, and for each $\pi\in S_n$, the two preimages of $\pi$ in $2.S_n^-$ are $+\pi$ and $-\pi$. According to~\cite[\S 2.7]{W2009}, the following relations hold:
    \begin{itemize}
        \item $[i,j]^2=-e$ for $i,j\in [n]$ with $i\neq j$.
        \item $[i,j]^{\pm \pi}=-[i^\pi,j^\pi]$ for odd permutation $\pi\in S_n$ and $i,j\in [n]$ with $i\neq j$.
        \item $[a_i,a_{i+1},\ldots, a_j]=[a_i,a_{i+1}][a_i,a_{i+2}]\cdots [a_i,a_j]$ for $i<j$.
    \end{itemize}

\begin{example}
    Let $G=2.S_4^-$, and let $M=\la [1,2], [1,2,3]\ra\cong 2.S_3^-$. Then $K:=\mathrm{Core}_G(M)=\{\pm e\}$, $G/K\cong S_4$ and $M/K\cong S_3$. Since $\{e,(1,4), (2,4), (1,2)(3,4)\}$ is an inverse-closed left transversal of $S_3$ in $S_4$, we know from Theorem~\ref{thm:1} that $M/K$ is a perfect code of $G/K$. Let $x=[1,2]$, $y=[3,4]$, and $Q=\la x\ra\cong C_4$. Then $Q\in \Syl_2(M)$ and $Q\la y\ra=\la x,y\mid x^4=e,\ y^2=x^2,\ x^y=x^{-1}\ra\cong Q_8$. Thus, we derive from Lemma~\ref{lm:cyclic} that $Q$ is not a perfect code of $G$, and so, by Lemma~\ref{lm:Z2023}, $M$ is not a perfect code of $G$.
\end{example}

To establish our next example, we need the following complete classification of subgroup perfect codes of $\SL_2(q)$.

\begin{lemma}\label{lm:SL}
Let $G=\SL_2( q)$ with prime power $q$. For even $q$, every subgroup of $G$ is a perfect code. For odd $q$, a subgroup $H$ of $G$ is a perfect code if and only if either $|H|_2=1$ or $|H|_2=|G|_2$.
\end{lemma}

\begin{proof}
For even $q$, since the Sylow $2$-subgroups of $\SL_2(q)$ are elementary abelian, the conclusion follows from Lemma~\ref{lm:elem}. Now assume that $q$ is odd. The sufficiency follows immediately from Lemma~\ref{lm:Z2023}. For the necessity, suppose for a contradiction that $H$ is a subgroup perfect code of $G$ with $1<|H|_2<|G|_2$. Let $Q\in \Syl_2(H)$, and let $K\leq G$ with $Q\leq K$ and $|K:Q|=2$. It can be read off from~\cite[Table~8.1]{BHD2013} that every Sylow $2$-subgroup of $G$ is generalized quaternion. Hence, each of $Q$ and $K$ is either cyclic or generalized quaternion. Therefore, Lemmas~\ref{lm:cyclic} and~\ref{lm:quaternion} imply that $Q$ is not a perfect code of $G$. Hence, by Lemma~\ref{lm:Z2023}, $H$ is not a perfect code of $G$, a contradiction. This completes the proof.
\end{proof}

\begin{example}\label{ex:SL}
    Let $q=p^f\geq 5$ for some odd prime $p$ and positive integer $f$, and let $G=\SL_2( q)$. Then $G$ has a maximal subgroup
    \[
        M=\Bigg\{\begin{pmatrix}
            1 & a\\
            0 & 1
        \end{pmatrix}\mid a\in \FF_q\Bigg\}\rtimes
        \Bigg\{\begin{pmatrix}
            b & 0\\
            0 & b^{-1}
        \end{pmatrix}\mid b\in \FF_q^*\Bigg\}\cong C_p^f\rtimes C_{q-1}.
    \]
    Noting that $|M|_2=(q-1)_2<(q^2-1)_2=|G|_2$, we obtain from Lemma~\ref{lm:SL} that $M$ is not a perfect code of $G$. However, letting $K=\mathrm{Core}_G(M)$, we have
    \[
        K=\ZZ(G)\cong C_2,\ \ M/K\cong C_p^f\rtimes C_{(q-1)/2},\ \ G/K=\PSL_2(q),
    \]
    and so, by~\cite{CLZ2025} (or Lemma~\ref{lm:PSL}), $M/K$ is a perfect code of $G/K$.\qed
\end{example}

\section{Normal subgroups and proof of Theorem~\ref{thm:6}}\label{sec:5}

In this section, we prove Theorem~\ref{thm:6}---that is, for each primitive group of type $\mathrm{HA}$, $\mathrm{HS}$, $\mathrm{HC}$, $\mathrm{TW}$, $\mathrm{SD}$ or $\mathrm{CD}$, every point stabilizer is a perfect code of the group. Since the same proof applies to quasiprimitive groups of those six types, we prove Theorem~\ref{thm:6} by establishing more general results in Lemmas~\ref{lm:4types} and~\ref{lm:2types}. The following result is a direct corollary of the Diamond Lemma (Lemma~\ref{lm:diamond}).

\begin{lemma}\label{lm:normal}
    Let $G$ be a group, let $N$ be a normal subgroup of $G$, and let $M$ be a subgroup of $G$.  If $M\cap N$ is a perfect code of $N$, then $M$ is a perfect code of $G$.
\end{lemma}

With this lemma, we immediately obtain the following two results, which together complete the proof of  Theorem~\ref{thm:6}.

\begin{lemma}\label{lm:4types}
    Let $G$ be a quasiprimitive group of type $\mathrm{HA}$, $\mathrm{HS}$, $\mathrm{HC}$ or $\mathrm{TW}$, and let $M$ be a point stabilizer. Then $M$ is a perfect code of $G$.
\end{lemma}

\begin{proof}
    Since $G$ is a quasiprimitive group of type $\mathrm{HA}$, $\mathrm{HS}$, $\mathrm{HC}$, or $\mathrm{TW}$, it has a regular normal subgroup $N$. This means that $G=NM$ and $N\cap M=1$. In particular, $N$ is an inverse-closed left transversal of $M$ in $G$. Then it follows from Theorem~\ref{thm:1} that $M$ is a perfect code of $G$.
\end{proof}

\begin{lemma}\label{lm:2types}
    Let $G$ be a quasiprimitive group of type $\mathrm{SD}$ or $\mathrm{CD}$, and let $M$ be a point stabilizer. Then $M$ is a perfect code of $G$.
\end{lemma}

\begin{proof}
    Letting $N=\Soc(G)$, we have $N=L^k$ for some  non-abelian characteristically simple group $L$ and integer $k\geq 2$ such that $M\cap N=\{(\ell,\ldots, \ell)\mid \ell\in L\}$. Then since $L^{k-1}$ is a complement of $M\cap N$ in $N$, we derive from Theorem~\ref{thm:1} that $M\cap N$ is a perfect code of $N$. Since $G$ is quasiprimitive, the normal subgroup $N$ is transitive, which means $G=MN$. Therefore, by Lemma~\ref{lm:normal}, $M$ is a perfect code of $G$.
\end{proof}

\section{$2$-local analysis, primitive groups of type $\mathrm{AS}$, and proofs of Theorems~\ref{thm:AS} and~\ref{thm:perfectcode}}\label{sec:6}

In this section, we prove Theorems~\ref{thm:AS} and~\ref{thm:perfectcode} in two subsections respectively.

\subsection{Type~$\mathrm{AS}$ with socle $\PSL_2(q)$}

This subsection is devoted to the proof of Theorem~\ref{thm:AS}. For convenience, we pose the following hypothesis.

\begin{hypothesis}\label{hypo:1}
Let $G$ be an almost simple group with socle $T=\PSL_2(q)$, where $q\geq 4$ is an odd prime power, and let $M$ be a maximal subgroup of $G$ such that $M\cap T$ is a maximal subgroup of $T$.
\end{hypothesis}

We start with the following result, which was proved in~\cite{CLZ2025}, with the proof relying on results from~\cite{MWWZ2020,Z2023}. We present an alternative proof that uses only preliminaries from Section~\ref{sec:2} and results from Section~\ref{sec:3}, not only for the sake of self-containment, but also to illustrate how the new criterion for determining subgroup perfect codes developed in Section~\ref{sec:3} can be applied.

\begin{lemma}[\cite{CLZ2025}]\label{lm:PSL}
    Let \( q \) be a prime power. Then a maximal subgroup \( M \) of \( \PSL_2( q) \) is not a perfect code if and only if one the following holds:
    \begin{enumerate}[\rm(a)]
        \item \( q > 7 \), \( q \equiv -1 \pmod{8} \), and \( M\cong D_{q-1} \);
        \item \( q > 9 \), \( q \equiv 1 \pmod{8} \), and \( M \cong D_{q+1} \).
    \end{enumerate}
\end{lemma}

\begin{proof}
    It is easy to verify by GAP~\cite{GAP4} that every maximal subgroup of \( \PSL_2( q) \) is a perfect code for $q<5$. Now we assume that $q\geq 5$. Let $G=\PSL_2( q)$ with a maximal subgroup $M$, and let $Q\in \Syl_2(M)$. If $q$ is a power of $2$, then it follows from Lemma~\ref{lm:Sylow2PSL} that $G$ has elementary abelian Sylow $2$-subgroups, and so, we conclude from Lemma~\ref{lm:elem} that $M$ is a perfect code of $G$.

    Next, we may assume that $q$ is odd. Then by Lemma~\ref{lm:Sylow2PSL}, each Sylow $2$-subgroup of $G$ is isomorphic to $D_{2n}$ for some $n\geq 2$, which combined with Corollary~\ref{cor:dihedral} shows that $Q$ is a not a perfect code of $G$ if and only if
    \begin{equation}
        \label{eq:Q}
        \text{there exists a cyclic $2$-subgroup $C\leq G$ such that $1<Q<C$.}
    \end{equation}
    Note by Lemma~\ref{lm:MaximalPSL} that $Q$ is cyclic only if one of the following holds:
    \begin{enumerate}[\rm(i)]
        \item $M\cong C_p^f\rtimes C_{(q-1)/2}$ and $Q\cong C_{(q-1)_2/2}$, where $q=p^f$ and $p$ is a prime;
        \item \label{enu:ii} $M\cong D_{q-1}$, $Q\cong C_2$, $q\geq 13$ and $q\equiv -1\pmod{4}$;
        \item \label{enu:iii} $M\cong D_{q+1}$, $Q\cong C_2$, $q\neq 9$ and $q\equiv 1\pmod{4}$.
    \end{enumerate}

    Assume first that $q\equiv 1\pmod{4}$. Then we see from Lemma~\ref{lm:Sylow2PSL} that each Sylow $2$-subgroup of $G$ is isomorphic to $D_{(q-1)_2}$. Hence, the maximal cyclic $2$-subgroup of $G$ is isomorphic to $C_{(q-1)_2/2}$. It follows from~\eqref{eq:Q} that $Q$ is not a perfect code of $G$ only if conditions in~\eqref{enu:ii} hold. If $q\equiv 5\pmod{8}$, then $|Q|=2=|C_{(q-1)_2/2}|$, and so we derive from~\eqref{eq:Q} that $Q$ is a perfect code of $G$. If $q\equiv 1\pmod{8}$, then $2=|Q|<|C_{(q-1)_2/2}|$. Thus, we obtain from Lemma~\ref{lm:involutionPSL} that there exists $C\cong C_{(q-1)_2/2}$ such that $1<Q<C$, which combined with~\eqref{eq:Q} gives that $Q$ is not a perfect code of $G$. Therefore, we conclude that when $q\equiv 1\pmod{4}$, $Q$ is not a perfect code of $G$ if and only if
    \[
        q>9, \ \ q\equiv 1\pmod{8},\ \text{ and }\, M\cong D_{q+1}.
    \]

    Assume next that $q\equiv -1\pmod{4}$. Then by Lemma~\ref{lm:Sylow2PSL}, each Sylow $2$-subgroup of $G$ is isomorphic to $D_{(q+1)_2}$, and so the maximal cyclic $2$-subgroup of $G$ is isomorphic to $C_{(q+1)_2/2}$. Since $C_{(q-1)_2/2}=1$, it follows from~\eqref{eq:Q} that $Q$ is not a perfect code of $G$ only if conditions in~\eqref{enu:iii} hold. Following a similar argument to the above paragraph, we assert that when $q\equiv -1\pmod{4}$, $Q$ is not a perfect code of $G$ if and only if
    \[
        q>7,\ \ q\equiv -1\pmod{8},\ \text{ and }\,  M\cong D_{q-1}.
    \]
    These together with Lemma~\ref{lm:Z2023} complete the proof.
\end{proof}

The above result leads to the following corollary.

\begin{corollary}\label{cor:G/Todd}
    Under Hypothesis~$\ref{hypo:1}$, if $|G/T|$ is odd, then $M$ is not a perfect code if and only if one the following holds:
    \begin{enumerate}[\rm(a)]
        \item \( q > 7 \), \( q \equiv -1 \pmod{8} \), and \( M\cap T \cong D_{q-1} \);
        \item \( q > 9 \), \( q \equiv 1 \pmod{8} \), and \( M\cap T \cong D_{q+1} \).
    \end{enumerate}
\end{corollary}

\begin{proof}
    Since $G=MT$, we have $M/(M\cap T)\cong MT/T=G/T$, which together with $|G/T|$ odd implies that $|M|_2=|M\cap T|_2$. Then the conclusion follows from Lemmas~\ref{lm:diamondConverse} and~\ref{lm:PSL}. 
\end{proof}

Next, we show that each maximal subgroup of \( \PGL_2( q) \) is a perfect code.

\begin{lemma}\label{lm:PGL}
    Let \( q \) be a prime power. Then each maximal subgroup of \( \PGL_2( q) \) is a perfect code.
\end{lemma}

\begin{proof}
    Let $G=\PGL_2(q)$. If $q$ is even, then it follows from Lemma~\ref{lm:Sylow2PSL} that $G$ has Sylow $2$-subgroups, which combined with Lemma~\ref{lm:elem} shows that each maximal subgroup of $G$ is a perfect code of $G$.

    Now assume that $q$ is odd. Let $M$ be an arbitrary maximal subgroup of $G$, and let $Q\in \Syl_2(M)$. We see from Lemma~\ref{lm:Sylow2PSL} that each Sylow $2$-subgroup of $G$ is isomorphic to $D_{2n}$ for some $n\geq 2$. Hence, we deduce from Corollary~\ref{cor:dihedral} that $Q$ is not a perfect code of $G$ if and only if there exists a cyclic $2$-subgroup $C\leq G$ such that $1<Q<C$. By Lemma~\ref{lm:MaximalPSL}, we see that $Q$ is cyclic only if $M\cong C_p^f\rtimes C_{(q-1)/2}$ and $Q\cong C_{(q-1)_2}$, where $q=p^f$ and $p$ is a prime.

    Assume first that $q\equiv 1\pmod{4}$. Then we see from Lemma~\ref{lm:Sylow2PSL} that each Sylow $2$-subgroup of $G$ is isomorphic to $D_{(q-1)_2}$. Hence, the maximal cyclic $2$-subgroup of $G$ is isomorphic to $C_{(q-1)_2}\cong Q$. Therefore, we conclude that $Q$ is not a perfect code of $G$.

    Assume next that $q\equiv 3\pmod{4}$. Then we obtain from Lemma~\ref{lm:PGL2q} that any Sylow $2$-subgroup of $\NN_G(Q)$ is isomorphic to $D_4$, which is elementary abelian. Hence, we conclude from Lemma~\ref{lm:elem} that $Q$ is a perfect code of $\NN_G(Q)$. Therefore, we obtain from Lemma~\ref{lm:Z2023} that $Q$ is also a perfect code of $G$. This completes the proof.
\end{proof}

With the Diamond Lemma, we can extend the above result as follows.

\begin{corollary}\label{cor:geqPGL}
    Under Hypothesis~$\ref{hypo:1}$, if $G \ge PGL_{2}(q)$, then $M$ is a perfect code of $G$.
\end{corollary}

\begin{proof}
    Under the right multiplication action of $G$ on $G/M$, the set of right cosets of $M$ in $G$, the subgroups $M\cap T$ and $M\cap \PGL_2(q)$ are point stabilizers of $T$ and $\PGL_2(q)$, respectively. Since $M\cap T$ is maximal in $T$, it follows that $T$ is primitive. As a consequence, $\PGL_2(q)$ is also primitive, which implies that $M\cap \PGL_2(q)$ is maximal in $\PGL_2(q)$. Hence, by Lemma~\ref{lm:PGL}, $M\cap \PGL_2(q)$ is a perfect code of $\PGL_2(q)$. Therefore, noting that $G=TM=\PGL_2(q)M$, the conclusion follows from Lemmas~\ref{lm:diamond}.
\end{proof}

The following two lemmas deal with almost simple groups with socle $\PSL_2(q)$ that have not been discussed above.

\begin{lemma}\label{lm:6.6}
    Under Notation~$\ref{nota:PGaL}$, suppose that Hypothesis~$\ref{hypo:1}$ holds, $f$ is even, and $G=T\rtimes\la \phi^k\ra$ for some integer $k$ such that $|\phi^k|$ is even. Then $M$ is a perfect code of $G$.
\end{lemma}

\begin{proof}
    Let $H=T\rtimes \la \phi^{f/2}\ra$. Since $| \phi^k|$ is even, we have $H\leq G$. Since $G=MT=MH$, we see from Lemma~\ref{lm:diamond} that it suffices to show that $M_H:=M\cap H$ is a perfect code of $H$. Note that $M_H$ is maximal in $H$ as $M_T:=M\cap T$ is maximal in $T$.
    
    Since $H=G\cap H=MT\cap H=(M\cap H)T=M_HT$, it follows from Lemma~\ref{lm:diamond} that, if $M_T$ is a perfect code of $T$, then $M_H$ is a perfect code of $H$. Now assume that $M_T$ is not a perfect code of $T$. As $f$ is even, we have $q=p^f\equiv 1\pmod{8}$. Then Lemma~\ref{lm:PSL} implies that $q>9$ and $M_T\cong D_{q+1}$. Let $Q\in \Syl_2(M_H)$.

    In this paragraph, we show that $Q\cong C_4$ and $Q\cap T\cong C_2$. Let $\FF_{q^2}^\times=\la \zeta\ra$. Define mappings
    \[
        x\colon v\mapsto v\zeta^{q-1},\ \  y\colon v\mapsto v^{p^{f/2}},\ \ z\colon v\mapsto v\zeta^{q+1}
    \]
    for $v \in \FF_{q^2}$. Let $N=\la x^\psi, y^\psi\ra$ and $N_T=\la x^\psi, (y^2)^{\psi}\ra$. Since $\ZZ(\GL(2,q))=\la z\ra$, we have $|x^\psi|=(q+1)/2$ and $|y^\psi|=4$. Moreover, it is direct to verify that $(x^\psi)^{(y^2)^\psi}=(x^\psi)^{-1}$. Thus, 
    \[
        N\cap T=N_T\cong D_{q+1}\ \text{ and }\ N/N_T\cong \la \phi^{f/2}\ra\cong H/T\cong  M_H/M_T.
    \]
    By~\cite[Table~8.1]{BHD2013}, there exists $g\in G$ such that $M_H^g=N$. Observe that $\la y^\psi\ra\in \Syl_2(N)$ with $\la y^\psi\ra\cong C_4$ and $\la y^\psi\ra\cap T\cong C_2$. Then it follows that $Q\cong C_4$ and $Q\cap T\cong C_2$.

    By Lemma~\ref{lm:C8Q8}, there does not exist $K$ such that $Q\leq K\leq H$ and  $K\cong C_8$ or $Q_8$. Therefore, we conclude from Lemma~\ref{lm:cyclic} that $Q$ is a perfect code of $H$, which combined with Lemma~\ref{lm:Z2023} implies that $M_H$ is a perfect code of $H$. This completes the proof.
\end{proof}

\begin{lemma}\label{lm:6.7}
    Under Notation~$\ref{nota:PGaL}$, suppose that Hypothesis~$\ref{hypo:1}$ holds, $f$ is even, and $G=T.\la \delta\phi^k\ra$ for some integer $k$ with $|\phi^k|$ even. Then $M$ is not a perfect code of $G$ if and only if $|\phi^k|_2=2$ and \( M\cap T \cong D_{q+1} \).
\end{lemma}

\begin{proof}
    To prove the necessity, we suppose that $|\phi^k|_2\neq 2$ or $M\cap T\not \cong D_{q+1}$ and show that $M$ is a perfect code of $G$. If $|\phi^k|_2=1$, then $\la\delta\ra\leq \la \delta\phi^k\ra$, and so $G\geq \PGL_2(q)$, which combined with Corollary~\ref{cor:geqPGL} indicates that $M$ is a perfect code of $G$. Assume that $|\phi^k|_2\geq 4$. Then $\la \phi^{f/2}\ra\leq \la \delta\phi^k\ra$, and so $H:=T.\la \phi^{f/2}\ra\leq G$. Since $M\cap T$ is maximal in $T$, we have $M\cap H$ maximal in $H$. Then Lemma~\ref{lm:6.6} asserts that $M\cap H$ is a perfect code of $H$. Hence, noting that $G=MT=MH$, we conclude by Lemma~\ref{lm:diamond} that $M$ is a perfect code of $G$. Now assume that $|\phi^k|_2=2$, in which case we have $M\cap T\not \cong D_{q+1}$. Since the condition $f$ is even yields that  $q=p^f\equiv 1\pmod{8}$, it follows from Lemma~\ref{lm:PSL} that $M\cap T$ is a perfect code of $T$. Thus, Lemma~\ref{lm:diamond} implies that $M$ is a perfect code of $G$. This completes the proof of the necessity.

    To prove the sufficiency, suppose that $|\phi^k|_2=2$ and \( M\cap T \cong D_{q+1} \). Let $Q\in \Syl_2(M)$. From $M/(M\cap T)\cong G/T \cong \la \delta \phi^k\ra$ we obtain
    \[
        |Q|=|M|_2=|\la \delta\phi^k\ra|_2|M\cap T|_2=2|D_{q+1}|_2=4 \ \text{ and }\ Q\cap T\cong C_2. 
    \]
    If $Q\cong C_2\times C_2$, then $T$ has a complement in $T.\la \delta\phi^{f/2}\ra$, contradicting Lemma~\ref{lm:non-split}. Consequently, $Q\cong C_4$. Then according to Lemma~\ref{lm:Q8}, there exists $K$ such that $Q\leq K\leq G$ and $K\cong Q_8$, which combined with Lemma~\ref{lm:cyclic} implies that $Q$ is not a perfect code of $G$. Therefore, we derive from Lemma~\ref{lm:Z2023} that $M$ is not a perfect code of $G$, completing the proof.
\end{proof}

We are now ready to prove Theorem~\ref{thm:AS}.

\begin{proof}[Proof of Theorem~$\ref{thm:AS}$]
    Assume first that $M\cap T$ is not maximal in $T$. Then it can be read off from~\cite[Tables~8.1 and~8.2]{BHD2013} that either $q\in \{7,9,11\}$ or $G=\PGL_2(q)$. For $q\in \{7,9,11\}$, computation in GAP~\cite{GAP4} shows that $M$ is not a perfect code of $G$ if and only if $G= M_{10}$ and $M\cong \AGL_1(5)$, which is equivalent to Theorem~\ref{thm:AS}\eqref{enu:1.5c} with $q=9$. If $G=\PGL_2(q)$, then Lemma~\ref{lm:PGL} asserts that $M$ is a perfect code of $G$. 
    
    Assume next that $M\cap T$ is a maximal subgroup of $T$. If $q$ is even, then Lemma~\ref{lm:PSL} shows that $M\cap T$ is a perfect code of $T$, and so $M$ is a perfect code of $G$ by Lemma~\ref{lm:diamond}. In the remainder of the proof, we assume that $q$ is odd and adopt Notation~\ref{nota:PGaL}. For odd $|G/T|$, the conclusion of the theorem follows from Corollary~\ref{cor:G/Todd}. If $|G/T|$ is even, then one of the following holds:
    \begin{enumerate}[\rm(i)]
        \item $G\geq \PGL_2(q)$;
        \item $G=T\rtimes\la \phi^k\ra$ for some integer $k$ such that $|\phi^k|$ is even;
        \item \label{enu:1.5iii} $G=T.\la \delta\phi^k\ra$ for some integer $k$ with $|\phi^k|$ even.
    \end{enumerate}
    In this case, we infer from Corollary~\ref{cor:geqPGL} and Lemmas~\ref{lm:6.6} and~\ref{lm:6.7} that $M$ is not a perfect code of $G$ if and only if~\eqref{enu:1.5iii} holds with $|\phi^k|_2=2$ and $M\cap T \cong D_{q+1}$, as in Theorem~\ref{thm:AS}\eqref{enu:1.5c}. 
\end{proof}

\subsection{Type~$\mathrm{AS}$ with socle $A_n$}

This subsection is devoted to the proof of Theorem~\ref{thm:perfectcode}.

For finite sets $A$ and $B$, denote by $\Fun(A,B)$ the set of functions from $A$ to $B$. Note that, if $B$ is a group, then $\Fun(A,B)$ is a group isomorphic to $B^{|A|}$. For a positive integer $t$ and a group $G$, let
    \[
        W=\Sym(G)\wr S_t=\Fun\big([t],\Sym(G)\big)\rtimes S_t,
    \]
where the action of $S_t=\Sym([t])$ on $\Fun\big([t],\Sym(G)\big)$ is defined by $f^{\pi}(i)=f(i^{\pi^{-1}})$ for each $f\in \Fun\big([t],\Sym(G)\big)$, $\pi\in S_t$ and $i\in [t]$. Then $W$ has a subgroup
\[
    D=\{f_g\in \Fun\big([t],R(G)\big) \mid g\in G\},
\]
where $f_g$ is defined by $f_g(i)=R(g)$ for each $i\in [t]$.

\begin{lemma}\label{lm:normalizer}
    Let $G$ be an abelian group, let $t$ be a positive integer, and let $W$ and $D$ be as above. Then for each $f\in \Fun\big([t],\Sym(G)\big)$ and $\pi \in S_t$ such that $f\pi\in \NN_{W}(D)$, there exists $\sigma \in \Aut(G)$ such that
    \[
        f(i)\in \sigma R(G) \ \text{ for each }\ i\in [t].
    \]
\end{lemma}

\begin{proof}
    It follows from $f\pi\in \NN_{W}(D)$ that, for each $f_g\in D$,
    \[
        (f\pi)^{-1}f_g(f\pi)=f_{g'} \ \text{ for some } g'\in G.
    \]
    This and the observation that $\pi$ centralizes $D$ imply that $f^{-1}f_gf=f_{g'}$, and so
    \[
        f(i)^{-1}R(g)f(i)=R(g') \ \text{ for each } i\in [t].
    \]
    As a consequence, we obtain
    \begin{align}
        &f(i)\in \NN_{\Sym(G)}(R(G)) \ \text{ for each } i\in [t],\label{eq:N}\\
        &f(i)f(j)^{-1}\in \mathsf{C}_{\Sym(G)}(R(G)) \ \text{ for each } i,j\in [t].\label{eq:C}
    \end{align}
    Since $G$ is abelian and $R(G)$ is regular on $G$,
    \begin{equation}\label{eq:CN}
        \mathsf{C}_{\Sym(G)}(R(G))=R(G)\ \text{ and }\ \NN_{\Sym(G)}(R(G))=\mathrm{Hol}(G)=\Aut(G)R(G).
    \end{equation}
    It follows from~\eqref{eq:N} and~\eqref{eq:CN} that, for each $i\in[t]$, there exists $\sigma_i\in\Aut(G)$ such that $f(i)\in\sigma_i R(G)$. Thus, noting that $R(G)$ is normalized by $\Aut(G)$, we conclude that
    \[
        f(i)f(j)^{-1}\in \sigma_iR(G) R(G) \sigma_j^{-1}=\sigma_i \sigma_j^{-1} R(G)\ \text{ for each } i,j\in [t].
    \]
    This combined with~\eqref{eq:C} and~\eqref{eq:CN} leads to
    \[
        \sigma_i=\sigma_j \ \text{ for each } i,j\in [t].
    \]
    Therefore, there exists $\sigma \in \Aut(G)$ such that $f(i)\in \sigma R(G)$ for each $i\in [t]$, as required.
\end{proof}

For the rest of the proof, we will use the following notations repeatedly.

\begin{notation}\label{nota:}
    Let $q$ be a prime power such that $q\equiv 3\pmod{4}$, let $n=(q^2-1)_2$, and let $n'=(q^2-1)_{2'}$. Let $V=\FF_{q^2}$, and let $\Omega=\mathbb{F}_{q^2}^{\times}=\langle \zeta \rangle$. Define $x,y\in \Sym(V)$ by letting
    \begin{equation}\label{eq:xy}
        x\colon v\mapsto \zeta^{n'} v\ \text{ and }\ y\colon v\mapsto v^q.
    \end{equation}
    As both $x$ and $y$ fix $0$, we sometimes identify $x$ and $y$ with their restrictions to $\Omega$.
\end{notation}

To prove Theorem~\ref{thm:perfectcode}, we need the following technical lemma.

\begin{lemma}\label{lm:g2inx}
    Under Notation~$\ref{nota:}$, for each $a\in \NN_{\Sym(\Omega)}(\langle x\rangle)$ with $a^2\in \la x,y\ra$, we have $a^2\in \langle x\rangle$.
\end{lemma}

\begin{proof}
We first embed $\NN_{\Sym(\Omega)}(\langle x\rangle)$ into a wreath product. Let $\psi:\langle x\rangle\times [n']\to \Omega$ be a mapping defined by
\begin{equation}\label{eq:psi}
    \psi\colon (x^k,i)\mapsto \zeta^{n'k+i}.
\end{equation}
This is a bijection and induces an isomorphism
\begin{equation}\label{eq:isophi}
    \varphi\colon \Sym(\Omega)\to \Sym(\langle x\rangle\times [n']),\ \ h\mapsto \psi h\psi^{-1}.
\end{equation}
It is readily seen from~\eqref{eq:xy} that
\begin{equation}\label{eq:rel}
    x^n=y^2=e,\ \ y^{-1}xy=x^q,
\end{equation}
and $\langle x\rangle$ acts semi-regularly on $\Omega$ with orbits
\[
    \langle \zeta^{n'}\rangle\zeta^1,\ \langle \zeta^{n'}\rangle\zeta^2,\ldots,\ \langle \zeta^{n'}\rangle\zeta^{n'}.
\]
Since $\NN_{\Sym(\Omega)}(\langle x\rangle)$ permutes these orbits, it follows from~\eqref{eq:psi} and~\eqref{eq:isophi} that $(\NN_{\Sym(\Omega)}(\langle x\rangle))^\varphi$ preserves the partition $\{\langle x\rangle\times \{1\},\ldots,\langle x\rangle\times \{n'\}\}$ of $\langle x\rangle\times [n']$, and so
\begin{equation}\label{eq:wreath}
    (\NN_{\Sym(\Omega)}(\langle x\rangle))^\varphi\leq \Sym(\langle x\rangle)\wr \Sym([n']).
\end{equation}

Next, we characterize $a^\varphi$ and $y^\varphi$ in terms of elements in $\Sym(\langle x\rangle)\wr \Sym([n'])$, as $a^\varphi$ and $y^\varphi$ are both in $(\NN_{\Sym(\Omega)}(\langle x\rangle))^\varphi$. Elements of $\Sym(\langle x\rangle)\wr \Sym([n'])$ can be written as $f\pi$, where $f\in \Fun\big([n'],\Sym(\langle x\rangle)\big)$ and $\pi\in \Sym([n'])$ such that
\begin{equation}\label{eq:wraction}
    (x^k,i)^{f\pi}=((x^k)^{f(i)},i^\pi)\ \text{ for each }\ (x^k,i)\in \langle x\rangle\times [n'].
\end{equation}
As introduced at the beginning of Section~\ref{sec:2}, $R(\la x\ra)$ denotes the right regular permutation representation of  $\la x\ra$. For each $k\in [n]$ define $f_{x^k}\in \Fun\big([n'],\Sym(\langle x\rangle)\big)$ by letting
\begin{equation}\label{eq:fxk}
    f_{x^k}(i)=R(x^k)\ \text{ for each } i\in [n'],
\end{equation}
and define
\[
    D=\{f_{x^k} \mid k\in [n]\}.
\]
Then it is straightforward to verify by~\eqref{eq:psi},~\eqref{eq:isophi},~\eqref{eq:wraction} and~\eqref{eq:fxk} that
\begin{equation}\label{eq:phi}
    (x^k)^\varphi= \psi x^k\psi^{-1}= f_{x^k}\ \text{ for each } k\in [n],
\end{equation}
which implies that $D=\langle x\rangle^\varphi$. Since $\varphi$ is an isomorphism, we obtain
\[
    (\NN_{\Sym(\Omega)}(\langle x\rangle))^\varphi
    =\NN_{(\Sym(\Omega))^\varphi}(\langle x\rangle^\varphi)
    =\NN_{\Sym(\langle x\rangle\times [n'])}(D).
\]
This together with~\eqref{eq:wreath} implies that
\[
    (\NN_{\Sym(\Omega)}(\langle x\rangle))^\varphi=\NN_{\Sym(\langle x\rangle\times [n'])}(D)\cap (\Sym(\langle x\rangle)\wr \Sym([n']))=\NN_{\Sym(\langle x\rangle)\wr \Sym([n'])}(D).
\]
Therefore, by Lemma~\ref{lm:normalizer}, for each $f\pi\in (\NN_{\Sym(\Omega)}(\langle x\rangle))^\varphi$, where $f\in \Fun\big([n'],\Sym(\langle x\rangle)\big)$ and $\pi\in \Sym([n'])$, there exists $\sigma\in \Aut(\langle x\rangle)$ such that $f(i)\in \sigma R(\langle x\rangle)$ for each $i\in [n']$. Then since $a,y\in \NN_{\Sym(\Omega)}(\langle x\rangle)$, we have $a^\varphi,y^\varphi\in (\NN_{\Sym(\Omega)}(\langle x\rangle))^\varphi$, and so
\[
   a^\varphi=f\pi\ \text{ and }\ y^\varphi=g\rho
\]
for some $f,g\in \Fun\big([n'],\Sym(\langle x\rangle)\big)$ and $\pi,\rho\in \Sym([n'])$ such that
\begin{equation}\label{eq:fg}
    f(i)\in \sigma R(\la x\ra)\ \text{ and }\ g(i)\in \tau R(\langle x\rangle)
\end{equation}
for some $\sigma,\tau\in \Aut(\langle x\rangle)$ independent of $i$.

To prove the lemma, suppose for a contradiction that $a^2\notin \langle x\rangle$. Then we see from~\eqref{eq:rel} and $a^2\in \langle x,y\rangle$ that $a^2\in \langle x\rangle y$, whence
\[
    a^2=x^sy \ \text{ for some } s\in [n].
\]
It follows that $(a^\varphi)^2=(x^s)^\varphi y^\varphi=f_{x^s}g\rho$ by~\eqref{eq:phi}. This together with $(a^\varphi)^2=f\pi f\pi=ff^{\pi^{-1}}\pi^2$ implies that
\[
    ff^{\pi^{-1}}=f_{x^s}g.
\]
Since $\sigma,\tau\in \Aut(\langle x\rangle)\leq \NN_{\Sym(\la x\ra)}(R(\la x\ra))$, we derive from~\eqref{eq:fxk} and~\eqref{eq:fg} that, for each $i\in [n']$,
\begin{align*}
    &ff^{\pi^{-1}}(i)=f(i)f^{\pi^{-1}}(i)=f(i)f(i^{\pi})\in \sigma R(\langle x\rangle)\sigma R(\langle x\rangle)= \sigma^2 R(\langle x\rangle),\\
    &f_{x^s}g(i)=f_{x^s}(i)g(i)\in R(\langle x\rangle)\tau R(\langle x\rangle)=\tau R(\langle x\rangle).
\end{align*}
Therefore, we conclude that
\begin{equation}\label{eq:ff}
    \sigma^2=\tau.
\end{equation}
It follows from~\eqref{eq:rel} that $(x^\varphi)^q=(y^\varphi)^{-1}(x^\varphi){y^\varphi}$. Then, noting $x^\varphi=f_x$ from~\eqref{eq:phi}, we obtain
\begin{equation}\label{eq:frho}
    f_x^q=(y^\varphi)^{-1}(x^\varphi){y^\varphi}=(g\rho)^{-1}f_xg\rho=(g^{-1}f_xg)^{\rho}.
\end{equation}
According to~\eqref{eq:fg}, $g(1^{\rho^{-1}})=\tau R(x^t)$ for some $t\in [n]$. Hence,
\begin{align*}
    (g^{-1}f_xg)^{\rho}(1)
    =(g^{-1}f_xg)(1^{\rho^{-1}})
    &=g(1^{\rho^{-1}})^{-1}f_x(1^{\rho^{-1}})g(1^{\rho^{-1}})\\
    &=R(x^{-t})\tau^{-1}R(x)\tau R(x^t)
    =R(x^{-t})R(x^{\tau})R(x^t)
    =R(x^{\tau}),
\end{align*}
which together with $f_x^q(1)=R(x^q)$ and~\eqref{eq:frho} gives that
\begin{equation}\label{eq:tauq}
    x^{\tau}=x^q.
\end{equation}
Write $x^\sigma=x^r$ for some $r\in [n]$. Then, by~\eqref{eq:ff} and~\eqref{eq:tauq}, $x^{r^2}=x^{\sigma^2}=x^{\tau}=x^q$. However, as $|x|=n=(q^2-1)_2$ is divisible by $4$, this yields
\[
    r^2\equiv q \pmod{4},
\]
contradicting $q\equiv 3\pmod{4}$. This completes the proof.
\end{proof}

Now we are ready to prove Theorem~\ref{thm:perfectcode}.

\begin{proof}[Proof of Theorem~$\ref{thm:perfectcode}$]
    We follow Notation~\ref{nota:} and let $Q=\langle x,y\rangle$. It follows from~\eqref{eq:rel} and $q\equiv 3\pmod{4}$ that $|Q|=2n=2(q^2-1)_2=|\AGL_2(q)|_2=|H|_2$, and so $Q\in \Syl_2(H)$. Hence, by Lemma~\ref{lm:Z2023}, it suffices to show that $Q$ is a perfect code of $G$. Noting that
    \[
        \Big(\frac{n}{2}-1-q\Big)_2= \big((q+1)_2-(q+1)\big)_2 = \big((1-(q+1)_{2'})(q+1)_2\big)_2\geq 2(q+1)_2=n,
    \]
    we have $y^{-1}xy=x^q=x^{-1+n/2}$. Since $n=2(q+1)_2\geq 2^3$, we deduce from Corollary~\ref{cor:a^2inQ} that $Q$ is a perfect code of $G$ if and only if $a^2\in \langle x\rangle$ for each $a\in \NN_G(Q)\setminus Q$ with $a^2\in Q$.

    Take an arbitrary $a\in \NN_G(Q)\setminus Q$ with $a^2\in Q$. One can verify by~\eqref{eq:rel} that $\la x\ra$ is the unique cyclic subgroup of order $n$ in $Q$, and hence,
    \[
        a\in \NN_G(Q)\leq \NN_G(\la x\ra)=\NN_{\Sym(V)}(\la x\ra).
    \]
    Since $0$ is the only fixed point of $\langle x\rangle$, we obtain that $a$ fixes $0$. Thereby, we may identify $a$ with its restriction to $\Omega=V\setminus \{0\}$, and so, $a$ can be viewed as an element of $\NN_{\Sym(\Omega)}(\langle x\rangle)$. Thus, we obtain from Lemma~\ref{lm:g2inx} that $a^2\in \langle x\rangle$, which completes the proof.
\end{proof}

\section{Wreath products and primitive groups of type~$\mathrm{PA}$}\label{sec:7}

In this section, we study primitive groups of type~$\mathrm{PA}$. We first provide the following results on perfect codes in direct product, semi-direct product and wreath product of groups.

\begin{lemma}\label{lm:product}
    Let $H$ be a group, and let $Q\leq H$. Then the following statements hold:
    \begin{enumerate}[\rm(a)]
        \item \label{enu:times} for a group $K$ and a subgroup $L$, the direct product $Q\times L$ is a perfect code of $H\times K$ if and only if $Q$ and $L$ are perfect codes of $H$ and $K$ respectively;
        \item \label{enu:rtimes} for a group $K$ such that $Q\rtimes K\leq H\rtimes K$, the subgroup $Q$ is a perfect code of $H$ if and only if $Q\rtimes K$ is a perfect code of $H\rtimes K$;
        \item \label{enu:wr} for a permutation group $K$ such that $Q\wr K\leq H\wr K$, if $Q$ is a perfect code of $H$, then $Q\wr K$ is a perfect code of $H\wr K$.
    \end{enumerate}
\end{lemma}

\begin{proof}
    (a) Suppose that $Q$ and $L$ are perfect codes of $H$ and $K$, respectively.  Then there exist inverse-closed left transversals $T$ of $Q$ in $H$ and $S$ of $L$ in $K$. It follows that $T\times S$ is an inverse-closed left transversal of $Q\times L$ in $H\times K$. Hence, Theorem~\ref{thm:1} implies that $Q\times L$ is a perfect code of $H\times K$.

    Conversely, suppose that $Q\times L$ is a perfect code of $H\times K$. Take an arbitrary $a\in H$ such that $a^2\in Q$ and $|Q|/|Q\cap Q^a|$ is odd. It follows that $a\in H\times K$, $a^2\in Q\times L$, and
    \[
        \frac{|Q\times L|}{|(Q\times L)\cap (Q\times L)^a|}
        =\frac{|Q\times L|}{|(Q\times L)\cap (Q^a\times L)|}=\frac{|Q\times L|}{|(Q\cap Q^a)\times L|}=\frac{|Q|}{|Q\cap Q^a|}
    \]
    is odd. Then we derive from Theorem~\ref{thm:1} that there exists $b=hk\in a(Q\times L)$, where $h\in H$ and $k\in K$, such that $b^2=e$, and thus $h^2=k^2=e$. Since $a\in H$, we have $h\in aQ$. Applying Theorem~\ref{thm:1} again, we conclude that $Q$ is a perfect code of $H$. Similarly, $L$ is a perfect code of $K$. This proves~\eqref{enu:times}.

    (b) This follows from Theorem~\ref{thm:1} and the observation that $T$ is an inverse-closed left transversal of $Q\rtimes K$ in $H\rtimes K$ if and only if $T$ is an inverse-closed left transversal of $Q$ in $H$.

    (c) This is an immediate consequence of~\eqref{enu:times} and~\eqref{enu:rtimes}.
\end{proof}

We remark that the converse of Lemma~\ref{lm:product}\eqref{enu:wr} does not hold, which will be shown in the rest of this section by a family of counterexamples. This indicates that the research of subgroup perfect codes of primitive groups of type~$\mathrm{PA}$ cannot be simply reduced to that of type~$\mathrm{AS}$. In fact, even if $G=H\wr S_n$ for some primitive group $H$ of type~$\mathrm{AS}$ and positive integer $n$, the point stabilizer $M=N\wr S_n$ in $G$ being a perfect code does not necessarily imply that the point stabilizer $N$ in $H$ is a perfect code (note that, however, if $N$ is a perfect code of $H$, then $M$ is a perfect
code of $G$ by Lemma~\ref{lm:product}\eqref{enu:wr}). The following is a general result for $n=2$ and will be used to provide the family of counterexamples.

\begin{theorem}\label{thm:wr}
    Let $G$ be a group, and let $H$ be a subgroup of $G$. If $|H|_2\leq 2$, then $H\wr S_2$ is a perfect code of $G\wr S_2$.
\end{theorem}

\begin{proof}
    If $|H|_2=1$, then each Sylow $2$-subgroup of $H$ is trivial, which combined with Theorem~\ref{thm:1} and Lemma~\ref{lm:2group} implies that $H$ is perfect code of $G$. In this case, the conclusion follows from Lemma~\ref{lm:product}\eqref{enu:wr}. For the rest of the proof, assume that $|H|_2=2$.

    Take an arbitrary involution $v\in H$, and let $Q= \la v\ra\wr S_2$. Then
    \begin{equation}\label{eq:Qv}
        Q= \{(e,e), (e,v), (v,e), (v,v),(e,e)(1\, 2), (e,v)(1\, 2), (v,e)(1\, 2), (v,v)(1\, 2)\}.
    \end{equation}
    Note that  $Q\in \Syl_2(H\wr S_2)$, as $|H\wr S_2|_2=2|H|_2^2=8$. Let $x=(v,e)(1\,2)$ and $y=(1\,2)$. It is straightforward to verify that
    \[
        Q=\la x,y\ra\cong D_8\ \text{ with }\  x^4=y^2=e \ \text{ and } x^y=x^{-1}.
    \]
    By Lemma~\ref{lm:Z2023} and Corollary~\ref{cor:Da^2inQ}, it suffices to show that, given an arbitrary $a\in\NN_{G\wr S_2}(Q)\setminus Q$ with $a^2\in Q$, all the following hold:
    \begin{enumerate}[\rm(i)]
        \item \label{enu:iwr} $a^2\in \langle x\rangle=\{(e,e),(v,v), (v,e)(1\,2), (e,v)(1\,2)\}$;
        \item \label{enu:iiwr} if $\la x,a\ra\cong C_8$, then $a^y=a^{-1}$;
        \item \label{enu:iiiwr} if $\la x, a\ra\cong Q_8$, then $[a,y]\in \la x^2\ra$.
    \end{enumerate}
    Note that every element of $G\wr S_2$ has the form $(a_1,a_2)\pi$ where $a_1,a_2\in G$ and $\pi\in S_2=\{e,(1\,2)\}$. We complete the proof in the following two cases.

    Assume first that $a=(a_1,a_2)(1\,2)$ for some $a_1,a_2\in G$. Then
    \[
        a^2=(a_1,a_2)(1\,2)(a_1,a_2)(1\,2)=(a_1,a_2)(a_1,a_2)^{(1\,2)}=(a_1,a_2)(a_2,a_1)=(a_1a_2,(a_1a_2)^{a_2^{-1}}).
    \]
    This combined with $a^2\in Q$ and~\eqref{eq:Qv} implies that $a^2\in \{(e,e),(v,v)\}$. Consequently, $|a|$ divides $4$, and $a^2\in \la x\ra$, as~\eqref{enu:iwr} requires. Moreover, since $|x|=4$ and $|a|$ divides $4$, we have $\la x,a\ra\not\cong C_8$, and so~\eqref{enu:iiwr} holds. Suppose that $\la x, a\ra\cong Q_8$. Since $x$ is an element of order $4$ in $Q_8$, we have $x^a=x^{-1}$, and so
    \[
        (e,v)(1\,2)=x^{-1}=x^a=(1\,2)(a_1^{-1},a_2^{-1})(v,e)(1\,2)(a_1,a_2)(1\,2)=(a_2^{-1}a_1,a_1^{-1}va_2)(1\,2).
    \]
    This implies that $a_1=a_2$. Hence,
    \[
        [a,y]=(1\,2)(a_1^{-1},a_2^{-1})(1\,2)(a_1,a_2)(1\,2)(1\,2)=(a_2^{-1}a_1,a_1^{-1}a_2)=(e,e)\in \la x^2\ra,
    \]
    proving~\eqref{enu:iiiwr}.

    Assume next that $a=(a_1,a_2)$ for some $a_1,a_2\in G$. Then
    \begin{align}
        & a^2=(a_1,a_2)^2=(a_1^2,a_2^2)\label{eq:a2xa}\\
        & x^a=(a_1^{-1},a_2^{-1})(v,e)(1\,2)(a_1,a_2)=(a_1^{-1}va_2,a_2^{-1}a_1)(1\,2).\label{eq:xa}
    \end{align}
    Since $\la x\ra$ is a characteristic subgroup of $Q$ and $a$ normalizes $Q$, we obtain $\la x\ra^a=\la x\ra$, and so $x^a=x$ or $x^a=x^{-1}$. Hence,~\eqref{eq:xa} leads to $(a_1^{-1}va_2,a_2^{-1}a_1)=(v,e)$ or $(e,v)$, which implies
    \[
        a_1=a_2\ \text{ or }\ a_1=va_2=a_2v.
    \]
    In either case, $(a_1^{-1}a_2,a_2^{-1}a_1)\in \{(e,e),(v,v)\}$ and $a_1^2=a_2^2$, which together with  $a^2\in Q$ and~\eqref{eq:a2xa} implies $a^2\in \{(e,e),(v,v)\}$. In particular,~\eqref{enu:iwr} holds, and $|a|$ divides $4$. Moreover, since $|x|=4$, we derive that $\la x,a\ra\not\cong C_8$, proving~\eqref{enu:iiwr}. Finally, since
    \[
        [a,y]=(a_1^{-1},a_2^{-1})(1\,2)(a_1,a_2)(1\,2)=(a_1^{-1}a_2,a_2^{-1}a_1)\in \{(e,e),(v,v)\}\subseteq \la x^2\ra,
    \]
    we conclude that~\eqref{enu:iiiwr} holds. This completes the proof.
\end{proof}

Recall from Lemma~\ref{lm:PSL} that a maximal subgroup of $\PSL_2(q)$ may fail to be a perfect code. In contrast, the following corollary of Theorem~\ref{thm:wr} reveals that  the corresponding maximal subgroups of $\PSL_2(q)\wr S_2$ are always perfect codes. This gives an infinite family of counterexamples to the converse of Lemma~\ref{lm:product}\eqref{enu:wr}.

\begin{corollary}\label{cor:wr}
    Let $T=\PSL_2(q)$ with prime power $q$, and let $N$ be a maximal subgroup of $T$. Then \( N\wr S_2 \) is a perfect code of \(T\wr S_2 \).
\end{corollary}

\begin{proof}
    By Lemma~\ref{lm:product}\eqref{enu:wr}, we are left to consider the case when $N$ is not a perfect code of $T$. By Lemma~\ref{lm:PSL}, we assume that one of the following holds:
    \begin{itemize}
        \item \( q > 9 \), \( q \equiv 1 \pmod{8} \), and \( N \cong D_{q+1} \);
        \item \( q > 7 \), \( q \equiv -1 \pmod{8} \), and \( N\cong D_{q-1} \).
    \end{itemize}
    In either case, $|N|_2=2$, and therefore, we conclude from Theorem~\ref{thm:wr} that $N\wr S_2$ is a perfect code of $T\wr S_2$, completing the proof.
\end{proof}

\textbf{Acknowledgements.} The authors would like to thank Ran Ju and Ju Zhang for their helpful discussions, which contributed to the improvement of this paper. Shouhong Qiao and Ning Su are supported by the National Natural Science Foundation of China (12471015) and the Natural Science Foundation of Guangdong Province (2025A1515012321). Shouhong Qiao also gratefully acknowledges the support of the China Scholarship Council and the hospitality of the School of Mathematics and Statistics at the University of Melbourne. Binzhou Xia and Sanming Zhou acknowledge the support of ARC Discovery Project DP250104965. Zhishuo Zhang is supported by the Melbourne Research Scholarship provided
by The University of Melbourne.

\end{document}